\documentclass[11pt]{amsart}
\usepackage{amssymb,amsmath}
\usepackage{graphicx}
\usepackage{datetime}

\newcommand{\R}{\ensuremath{\mathbb R}}
\newcommand{\N}{\ensuremath{\mathbb N}}

\theoremstyle{plain}
\newtheorem{theorem}{Theorem}[section]

\newtheorem{proposition}[theorem]{Proposition}

\newtheorem{exam}[theorem]{Example}
\newtheorem{defn}[theorem]{Definition}
\newtheorem{rem}[theorem]{Remark}
\newtheorem{rems}[theorem]{Remarks}

\theoremstyle{definition}

\numberwithin{equation}{section}

\begin{document}

\title{Piecewise smooth systems near a \\
co-dimension 2 discontinuity manifold: \\
can one say what should happen?}
\author{Luca Dieci}
\address{School of Mathematics \\ Georgia Tech \\
Atlanta, GA 30332 U.S.A.}
\email{dieci@math.gatech.edu}
\author{Cinzia Elia}
\address{Dipartimento di Matematica, Univ. of Bari, I-70100, Bari,
Italy}
\email{cinzia.elia@uniba.it}
\thanks{This work was begun while the second author was
visiting the School of Mathematics of Georgia Tech, whose hospitality is
gratefully acknowledged, and the author would also like to acknowledge the
support of the GNCS for the airfare.
The first author also gratefully acknowledges the support provided by
a Tao Aoqing Visiting Professorship at Jilin University, Changchun (CHINA)}



\begin{abstract}
We consider
a piecewise smooth system in the neighborhood of
a co-dimension 2 discontinuity manifold $\Sigma$ (intersection of two co-dimension 1 manifolds).  
Within the class of Filippov solutions, if $\Sigma$ is attractive,
one should expect solution trajectories to {\emph{slide}} on $\Sigma$.  It is well
known, however, that the classical
Filippov convexification methodology does not render a uniquely defined sliding vector
field on $\Sigma$.  The situation is further complicated by the possibility that, 
regardless of how sliding on $\Sigma$ is taking place, during sliding
motion a trajectory encounters so-called {\emph{generic first order exit points}}, where
$\Sigma$ ceases to be attractive.  

In this work, we attempt to understand what behavior one should
expect of a solution trajectory near $\Sigma$ when $\Sigma$ is attractive,
what to expect when $\Sigma$ ceases to be attractive (at least, at generic exit points), 
and finally we also contrast and compare the behavior of some
regularizations proposed in the literature,
whereby the original piecewise smooth system is replaced --in a neighborhood of
$\Sigma$-- by a smooth differential system.

Through analysis and experiments in $\R^3$ and $\R^4$,
we will confirm some known facts, and provide some important
insight: (i) when $\Sigma$ is attractive, a solution trajectory indeed does remain near
$\Sigma$, viz. sliding on $\Sigma$ is an appropriate idealization (of course, in general,
one cannot predict which sliding vector field should be selected);
(ii) when $\Sigma$ loses attractivity (at first order exit conditions), a typical
solution trajectory leaves
a neighborhood of $\Sigma$; (iii) there is no obvious way to
regularize the system so that the regularized trajectory will remain near $\Sigma$
as long as $\Sigma$ is attractive, and so that it will be leaving (a neighborhood of)
$\Sigma$ when $\Sigma$ looses attractivity.

We reach the above conclusions by considering exclusively the given piecewise smooth
system, without superimposing any assumption on what kind of dynamics near $\Sigma$
(or sliding motion on $\Sigma$) should have been taking place.  The only datum for us
is the original piecewise smooth system, and the dynamics inherited by it.
\end{abstract}

\maketitle

\pagestyle{myheadings}
\thispagestyle{plain}
\markboth{LUCA DIECI AND CINZIA ELIA}{WHAT SHOULD HAPPEN?}

\section{Introduction}\label{Intro_section}
Consider the following piecewise smooth (PWS) system:
\begin{equation}\label{GeneralPWS}
\dot x=f(x)\ ,\,\ f(x)=f_i(x)\ ,\,\ x\in R_i\ ,\, i=1,2,3,4,
\end{equation}
for $t\in [0,T]$, and where, for $i=1,2,3,4$,
$R_i\subseteq \R^n$ are open, disjoint and connected sets, and
$\R^n=\overline{\bigcup_{i} R_i}$.  
System \eqref{GeneralPWS} is subject to initial condition $x(0)=x_0$, prescribed
in one of the regions $R_i$'s.  In \eqref{GeneralPWS}, for any $i=1,2,3,4$,
each $f_i$ is smooth in $R_i$, so that there is a classical solution in
each region $R_i$, but the solution is not properly defined on the
boundaries of these regions.  We assume that these regions
are separated (locally) by an implicitely defined smooth manifold
$\Sigma$ of co-dimension $2$.  That is, we have
\begin{equation}\label{Sigma}
\Sigma=\{x\in \R^n \ :\, h(x)=0 \,,\quad h:\ \R^n \to \R^2 \}\,,
\end{equation}
and for all $x\in \Sigma$:
$h(x)=\begin{bmatrix}h_1(x) \\ h_2(x)\end{bmatrix}$, $\nabla h_j(x)\ne 0$,
$h_j \in {\mathcal C}^k, \ k\ge 2$, $j=1,2$,
and $\nabla h_1(x), \nabla h_2(x)$, are linearly independent on (and
in a neighborhood of) $\Sigma$.
It is useful to think of $\Sigma=\Sigma_1\cap \Sigma_2$, where 
$\Sigma_1=\{ \ x :\ h_1(x)=0\}$,
and $\Sigma_2=\{\ x:\ h_2(x)=0\}$, are co-dimension 1 manifolds.
Finally, without loss of generality we will henceforth use the
following labeling of the four regions $R_i$, $i=1,2,3,4$:
\begin{equation}\begin{split}\label{RegionsCod2}
& R_1:\quad \text{when} \quad h_1<0\ ,\, h_2<0\ ,\qquad 
R_2: \quad \text{when}\quad h_1<0\ ,\, h_2>0\ , \\
& R_3: \quad \text{when} \quad h_1>0\ ,\, h_2<0\ ,\qquad 
R_4: \quad \text{when}\quad h_1>0\ ,\, h_2>0\ .
\end{split}\end{equation}
For later use, we will also adopt the notation
$\Sigma_{1,2}^\pm$ to denote the set of 
points $x \in \Sigma_{1}$ or $\Sigma_2$, for which we also have
$h_{2}(x)\gtrless 0$ or $h_{1}(x)\gtrless 0$.  E.g., 
$\Sigma_{1}^+=\{x\in \Sigma_1 \,\,\ \text{and} \,\,\ h_2(x)>0\}$.
Finally, we will denote with
\begin{equation}\label{Wij}
w^i_j(x)=\nabla h_i(x)^\top f_j(x)\ ,\,\, i=1,2,\ j=1,2,3,4\ ,
\end{equation}
the projections of the vector fields in the normal directions to the manifolds.

For \eqref{GeneralPWS}, a classical solution in general cannot exist on the
boundaries of the given regions, and several concepts of generalized solution have been
proposed during the years (see \cite{Cortes} for a beautiful exposition on different
solutions concepts).  We will restrict attention to Filippov solutions,
\cite{Filippov}, consisting of absolutely continuous functions whose derivative
is in the convex hull of the neighboring vector fields almost everywhere.

\subsection{Co-dimension 1}
In the case of one single discontinuity manifold of co-dimension 1, Filippov methodology
has provided a widely accepted mathematical framework to understand motion on the discontinuity
surface.  

Consider a discontinuity manifold $\Sigma=\{x\in \R^n:\ h(x)=0\ ,\,\ h:\R^n\to \R\ \}$, separating
two regions $R_1$ (where $h(x)<0$)  and $R_2$ (where $h(x)>0$),
with respective vector fields $f_1$ and $f_2$.  Assuming that
$\Sigma$ is attractive, a condition that is satisfied when
$$\nabla h^T f_1 >0 \quad \text{and} \quad \nabla h^T f_2 <0\ ,\,\ x\in \Sigma \ ,$$
then {\emph{sliding motion}} on $\Sigma$ takes place with vector field
\begin{equation}\label{FiliCod1}
f_{\text{F}} := 
(1-\alpha)f_1+\alpha f_2\ ,\qquad \alpha=\frac{\nabla h^T f_1}{\nabla h^T (f_1-f_2)}\ .
\end{equation}
Filippov theory provides also {\emph{first order exit conditions}}: if one of
$f_1$ or $f_2$ (but not both) become tangent to $\Sigma$, then $\alpha=0$ or $1$, $\Sigma$
loses attractivity, and
the solution trajectory generically will leave $\Sigma$ tangentially (and smoothly)
to enter in $R_1$ with
vector field $f_1$, or in $R_2$ with vector field $f_2$.
Furthermore, it has been understood for a long time that the limiting behavior of the iterates
obtained with Euler method near $\Sigma$ leads to the selection of the Filippov sliding
vector field itself (e.g., see \cite{Utkin.2}, Chapter 3, Section 1.1) 
whenever $\Sigma$ is attractive, and that the Euler
iterates leave a neighborhood of $\Sigma$ when $\Sigma$ loses attractivity. 

\subsection{Co-dimension 2}
However, when $\Sigma$ is of the form \eqref{Sigma}, an obvious lack of uniqueness (in general) arises
in the construction of a Filippov vector field sliding on $\Sigma$.
In fact, on $\Sigma$, Filippov methodology now
leads to the requirement that the vector field satisfies the following (for positive
values of $\lambda_1,\dots, \lambda_4$):
\begin{equation}\label{FiliCod2}\begin{split}
f_{\text{F}} & := \lambda_1f_1+\lambda_2 f_2+\lambda_3 f_3+\lambda_4 f_4 \ ,\quad \text{where} \\
& \begin{bmatrix} w^1_1 & w^1_2 & w^1_3 & w^1_4 \\ w^2_1 & w^2_2 & w^2_3 & w^2_4 \\ 1 & 1& 1& 1 \end{bmatrix}
\begin{bmatrix} \lambda_1 \\ \lambda_2 \\ \lambda_3 \\ \lambda_4 \end{bmatrix} = 
\begin{bmatrix} 0 \\ 0 \\ 1 \end{bmatrix}\ ,
\end{split}\end{equation}
which is clearly an underdetermined system of equations.  Indeed, even when $\Sigma$ is attractive,
in general \eqref{FiliCod2} has a one parameter family of solutions and hence
of possible Filippov sliding vector fields.

\begin{rems}$\,$
\begin{itemize}
\item[(i)]
Characterization of attractivity of $\Sigma$ in the present co-dimension 2 case is considerably
more elaborate than in the case of co-dimension 1.  We will assume that $\Sigma$ is either
{\emph{attractive by subsliding}}
(see \cite{DiElLo} ) or {\emph{attractive by spiralling}} (see \cite{DieciSpiral}), in the
form recalled below in Definition \ref{AttrSigma}.
\item[(ii)]
In the present case,  the general lack of uniqueness is not resolved by considering the limiting
behavior of the Euler iterates near $\Sigma$.  In fact, as already noted in \cite{Filippov}, 
the limit of the Euler iterates selects one specific element of the one-parameter family 
of Filippov solutions.
\end{itemize}
\end{rems}

\begin{defn}\label{AttrSigma}
$\Sigma$ (or a portion of it) is {\emph{attractive upon sliding}} if it is reached in finite time
by solution trajectories for any given nearby initial condition, and further there is sliding motion
towards $\Sigma$ along at least one of the  $\Sigma_{1,2}^+$.   When there is sliding motion towards
$\Sigma$ along all of the $\Sigma_{1,2}^\pm$'s then we say that $\Sigma$ is {\emph{nodally attractive}}.
$\Sigma$ (or portion of it) is {\emph{attractive by spiralling}} if $\Sigma$ is reached in finite time
by trajectories for any given nearby initial condition, and there is clockwise or counter
clockwise motion around $\Sigma$ (and no sliding on $\Sigma_{1,2}^\pm$) for the
functions $h_1(x)$ and $h_2(x)$.
\end{defn}

When $\Sigma$ is attractive, in the literature there have been at least two systematic proposals
to select the coefficients $\lambda_i$'s in \eqref{FiliCod2},
leading to sliding vector fields of Flilippov type on $\Sigma$: the 
{\emph{bilinear}} and the {\emph{moments}} vector fields.  The former has been extensively
studied,  see \cite{AlexSeid1,AlexSeid2, DiElLo, Jeffrey1} for example,  
and it consists in choosing the sliding vector field in the form
\begin{equation}\label{bilinear}\begin{split}
f_B&:=(1-\alpha)[(1-\beta)f_1+\beta f_2]+\alpha[(1-\beta)f_3+\beta f_4]\ ,\\
(\alpha,\beta):& \, 
\begin{bmatrix} w^1_1 & w^1_2 & w^1_3 & w^1_4 \\ w^2_1 & w^2_2 & w^2_3 & w^2_4 \end{bmatrix}
\begin{bmatrix} (1-\alpha)(1-\beta) \\ (1-\alpha)\beta \\ \alpha(1-\beta) \\ \alpha\beta \end{bmatrix} = 
\begin{bmatrix} 0 \\ 0 \end{bmatrix}\ .
\end{split}\end{equation}
The moments method has been recently introduced in \cite{DiDif2} and consists in solving the linear
system in \eqref{FiliCod2}, by appending to it the extra relation
\begin{equation}\label{moments}
\begin{bmatrix} d_1 & -d_2 & -d_3 & d_4 \end{bmatrix}
\begin{bmatrix} \lambda_1 \\ \lambda_2 \\ \lambda_3 \\ \lambda_4 \end{bmatrix} = 0\ ,\quad d_i=\|w_i\|_2\ ,
\,\ i=1,2,3,4\ .
\end{equation}
The two choices above generally lead to distinct solution
trajectories, a chief difference between them being the
behavior of the bilinear and moments trajectories at so-called 
{\emph{generic tangential exit points}}.  These were introduced in \cite{DiElLo}, and
are values of $x\in \Sigma$, where one (and just one) of the sliding vector fields 
on $\Sigma_1^\pm$ or $\Sigma_2^\pm$ is itself tangent to $\Sigma$ ({\emph{exit vector field}}).   
The moments vector field automatically aligns with the exit vector field, whereas the bilinear
vector field does not.

An important question, and one which we will try to answer in this work is:
what should happen to trajectories of \eqref{GeneralPWS} 
in a neighborhood of $\Sigma$, when $\Sigma$ loses attractivity (according to
generic first order exit conditions)?

\subsection{Regularize}
If {\emph{natura non facit saltum}}\footnote{Linnaeus' Philosophia Botanica, 1751}, then
\eqref{GeneralPWS} is either a description of an innatural system or a wrong 
model.  Probably, whenever it is set forth, it is neither innatural nor wrong, though it
may be a little of both things, perhaps because the model should be complemented by
some missing information (the 20 years old exposition of Seidman in \cite{Seid2} is still
worthwhile reading).   Regardless, the above 18th century
motto suggests considering a regularized version of \eqref{GeneralPWS}, by
replacing it with a smooth differential system.  Of course, we must assume that we do not
have knowledge of where \eqref{GeneralPWS} comes from, if from anywhere at all,
otherwise we should surely use this knowledge.  However, in the absence of further insight,
it is not obvious how one should globally regularize the system, and
several possibilities for globally regularizing the system have been proposed
in the literature.  

Arguably, the most studied regularization techniques are what we may call the 
``singular perturbation'' and ``sigmoid blending'' techniques.
The paper \cite{SotomayorTeixeira} was the first seminal work on the singular perturbation
approach, followed by the more recent works \cite{Llibre.Silva.Teixeira}, as well as
\cite{DiGu}, \cite{GuHairer}, \cite{Jeffrey1} and, in the context of gene regulatory networks, 
\cite{Plahte.Kjoglum_2005}. 
The first systematic exposition of blending was the
beautiful work \cite{AlexSeid1}.  But, in the end, these techniques are all rather
similar and amount to a regularization of the bilinear vector field.
This can be done locally, just in a neighborhood
of $\Sigma$, say using a cutoff function (as we do in Section \ref{Br_section}), or more
globally, perhaps through use of hyperbolic $\arctan$ functions to connect different vector fields.

In this work, we are exclusively concerned with local regularizations,
hence those that alter the given problem only in a neighborhood of $\Sigma$.
In this case, 
the above mentioned regularization proposals share some common traits, the most important ones being that,
outside of a neighborhood of $\Sigma$, the regularized vector field effectively reduces to the
original vector fields, and that the
regularization depends on a small parameter (or several small parameters) in such a way that as
the parameter(s) go to $0$, the neigborhood collapses onto $\Sigma$.  
The first fact is surely a reasonable property, 
since, away from $\Sigma$, there is a well defined smooth vector field depending
on where the trajectory is, be it one of the original $f_{1,2,3,4}$ or one of the Filippov
sliding vector fields on the surfaces $\Sigma_{1,2}$.   The second fact may be a bit more
controversial, since the regularized trajectory will typically select a specific sliding motion on 
$\Sigma$ (when $\Sigma$ is attractive),
as the neighborhood collapses onto $\Sigma$; however, as it is well understood,
and as we will also see, different regularizations do behave differently.  
As noted by Utkin ( \cite{Utkin.2} ), given that in a
neighborhood of $\Sigma$ there are non-unique dynamics, the inherited dynamics on $\Sigma$
(sliding motion) does depend on the choice of regularization. 
But, this being the case, is there an appropriate way to evaluate different
regularizations?  To answer this question, we must first
decide how we should evaluate dynamics.  

\subsection{Evaluate dynamics}
The first observation  
is that when we evaluate the dynamics of \eqref{GeneralPWS} we should
distinguish between the two cases (i) and (ii) below.
\begin{itemize}
\item[(i)] 
The PWS smooth system \eqref{GeneralPWS} is just a ``convenient'' formalism: there is
a ``true'' smooth system, defined globally, but
it is simpler to replace it by a PWS one.  For example, this is the case for
problems arising in gene regulatory networks (\cite{Machina.Edwards}, \cite{DelBuono.Elia.Lopez}).
In these cases, one knows what is the desired behavior in a neighborhood of $\Sigma$:
it is the behavior of the original problem!  
However, one must be careful in replacing the true system by \eqref{GeneralPWS}, since
it is not clear that the dynamics of the purely PWS system reflect
the dynamics of the original problem; this was already noted in works such as 
\cite{Polynikis.Hogan.diBernardo} and \cite{Seid2} relative to a discontinuity surface of co-dimension 1.  
But, doesn't this discrepancy mean that representing the original problem as a PWS
one was not an appropriate modeling simplification in the first place?
\item[(ii)]  
The problem arises as piecewise smooth problem, or we do not have sufficient
knowledge of an underlying ``true'' problem (if any);  for example, in bang-bang
control, or in dry friction models.
In our opinion, in these cases when there is no (knowledge of an) underlying
``true'' dynamics that one is trying to recover,
the choices we make must be consistent with the PWS formulation.   
This is the case in which we are interested.
\end{itemize}

The above consideration (ii) leads us to restrict to the model
\eqref{GeneralPWS} as {\bf the system} we are given and it is this system that
we will study.   
This realization motivates us (and it has motivated us for several years) to
look at the global properties of the discontinuity surface, and to decide on
what is appropriate based on whether or not the surface attracts the
dynamics of the PWS system for initial conditions off the discontinuity surface itself.
In our opinion, it is not easy to justify
studying \eqref{GeneralPWS} by assuming a specific form of an
underlying system from where \eqref{GeneralPWS} arises as some form of limiting process.  
In other words,
whereas it is surely legitimate (and by no means trivial)
to study the limiting behavior of a specific choice
of regularized vector field defined in a neighborhood of $\Sigma$, this may actually 
have a restricted scope of applicability compared to \eqref{GeneralPWS}.   
As already noted by Utkin (see \cite[p.44]{Utkin}), once one has
replaced \eqref{GeneralPWS} with a smoothed version of it,
the stability of sliding motion on $\Sigma$ is inherited by
that of the dynamics of the regularized field.  But, is this consistent with
the formulation \eqref{GeneralPWS}?

For the given PWS system \eqref{GeneralPWS}, we believe that
it is its dynamics near the discontinuity manifold that determines the appropriate
behavior of a trajectory.   This dynamics is what we will try to capture in this work.
Moreover, our main interest is in
the case when the discontinuity manifold transitions from being attractive to not attractive for
the trajectories of
the original discontinuous system: in this case, will (or should) a trajectory (even an ideal trajectory,
sliding on the manifold) feel this loss of attractivity, and hence leave a neighborhood of $\Sigma$?

Therefore, without assuming any form of idealized motion on $\Sigma$, we can reformulate our 
main task as follows: 

\centerline{{\emph{
``how should we evaluate the dynamics of \eqref{GeneralPWS} in a neighborhood of $\Sigma$?''}}}

\noindent
In principle, one may want to do this by studying the dynamics of a regularized problem,
and we have already mentioned some possibilities, such as sigmoid blending and singular perturbation
techniques.  E.g., see \cite{AlexSeid1, DiGu, GuHairer, Jeffrey1, Llibre.Silva.Teixeira, SotomayorTeixeira}.  
We emphasize once more that these choices (as noted by Alexander and Seidman for blending,
and by Teixeira et alia. for singular parturbation) remove the
ambiguity of how sliding on $\Sigma$ occurs, but these choices are effectively modeling
assumptions, and we should ask if they render
a behavior of the dynamics on $\Sigma$ that is consistent with that of \eqref{GeneralPWS}.

Other possibilities have also been set forth in the literature, see \cite{AlexSeid1} for
further references.

\begin{itemize}\label{Possibilities}
\item[(a)]
Euler broken line approximation.  This is simple to do, and it consists in replacing
\eqref{GeneralPWS} by a Euler method approximation with constant stepsize, call
it $\tau$.   We have experimented extensively with this technique, see below.
[The eventuality, of probability 0, that an Euler iterate lands exactly on
a discontinuity surface  
can be handled in different ways; e.g., by
randomly selecting one of the neighboring vector fields, or by retaining the last vector
field used.] 
\item[(b)]
Hysteresis (or delay) approximation.  This approach appears 
in \cite{Utkin.2} for a surface of co-dimension 1.
For the case of $\Sigma$ of co-dimension 2,
it has been studied first in \cite{AlexSeid2} and then in \cite{DiDif1}, always in the
case of nodally attractive $\Sigma$.  The idea here
is that one has a region $U_\epsilon$ around $\Sigma$ (called a chatterbox in \cite{AlexSeid2}),
and uses the same vector field, say  $f_1$, not only in the region $R_1$, but until the
boundary of $U_\epsilon$ in a different region is reached; at that point, a switch to
the appropriate vector field in the new region is performed.  The rationale for this approach
is that one does not notice immediately that a discontinuity surface is reached, but there
is a ``delay'' in appreciating this fact.
\item[(c)]
Replacing \eqref{GeneralPWS} with a stochastic DE of Ito type:
$\dot x=f_i(x)dt+\epsilon dw_t$.  Again we note that there is zero probability of landing
exactly on the discontinuity surface(s).  The interesting feature of this approach is that
it is bound to sample different vector fields around $\Sigma$.  The 
disadvantage is that it makes quantitative predictions possible only in 
a statistical sense.
\end{itemize}

With the exception of (c),  the other choices above effectively replace the original PWS with another
deterministic dynamical system, possibly discrete (as in the case of Euler method).  But, unfortunately,
these new systems have their own dynamics, and it is unclear whether or not these are consistent with that
of \eqref{GeneralPWS}.  See Section \ref{Numerics_section} for ample illustration of this fact.
At the same time, each of the cases above has some distinguished features, that
we will attempt to retain.

\subsection{Proposal and Plan} 
Our proposal to evaluate the dynamics
of \eqref{GeneralPWS} in a neighborhood of $\Sigma$ is to: 

\centerline{{\emph{``consider the Euler iterates
with random steps,''}}}

\noindent
with steps uniformly distributed about a reference stepsize.
In other words, we want to retain the
simplicity of looking at Euler iterates, but aim at retaining a certain amount of
randomness in the process to avoid getting trapped by the purely Euler dynamics.
Furthermore, in this work we will restrict to well-scaled vector fields $f_i$,
$i=1,\dots, 4$, with none of the $w^i_j$'s in \eqref{Wij} exceeding $1$ in absolute value.  
The reason for this is to avoid the numerical trajectory going too far away
from $\Sigma$, and at the same time to attempt retaining the flavor of a hysteretical trajectory.

We will complement our experiments made with the above strategy, by also using
other approaches.  For example, Euler method with constant stepsizes, singular perturbation
regularizations, and also numerical integration of \eqref{GeneralPWS}
performed with variable stepsize integrators.  

We emphasize that our goal is to reach some conclusions insofar as 
what should happen to a solution trajectory of
\eqref{GeneralPWS} in a (small) neighborhood of $\Sigma$.  We are not
comparing methods, or different recipes of ideal sliding motion, but
simply trying to evaluate the dynamics of \eqref{GeneralPWS} in the most plausible
and honest way we can think of, without superimposing on \eqref{GeneralPWS} any extra
modeling assumption.

Finally, one more caveat.  Our examples are all of systems in $\R^3$ and $\R^4$, and the
discontinuity surfaces $\Sigma_1$ and $\Sigma_2$ are
planes; this makes it easier to visualize and
understand things.  Higher dimensional state space, and non-planar discontinuity surfaces,
can surely bring new phenomena into play, but we have no reason to suspect that the basic
picture that emerges in our study, with the dichotomy between attractivity and lack of
attractivity of $\Sigma$, will be modified substantially.  

The remainder of this work is structured as follows.  In Section \ref{Br_section}, we consider
a prototypical regularization of the bilinear vector field \eqref{bilinear}, and give sufficient
(and sharp) conditions guaranteeing that the regularized solution converges to a sliding solution
on $\Sigma$ according to \eqref{bilinear}.  In Section \ref{Euler_section}, we give some details
of how we implemented the above mentioned proposal, particularly of what practical criteria we adopted
to detect ``exiting'' from a neighborhood of $\Sigma$.  In Section \ref{Numerics_section},
we illustrate through several examples the different things that can happen, and their
dependence on the adopted simulation choice.

\section{Bilinear interpolant regularization: solutions behavior and
fast slow dynamics}\label{Br_section}

Space regularizations are often employed in literature as an analytical mean to model
the switching mechanism of a discontinuous system, see 
\cite{AlexSeid2, DiGu, GuHairer, Jeffrey1, Llibre.Silva.Teixeira, SotomayorTeixeira, Utkin}.  
Typically, these regularizations are one parameter families of vector fields with different time
scales in a neighborhood of the sliding surface 
$\Sigma$, namely a slow dynamics tangent to $\Sigma$ and a fast one normal to $\Sigma$. 
In what follows, we consider regularizations for Filippov discontinuous systems  
as in \eqref{GeneralPWS}, but other approaches are available in the literature for 
non-smooth systems that are not of Filippov's type, for example control systems with nonlinear 
control, as in \cite{Utkin, Utkin.2}. When the regularization parameter  goes to zero, the 
regularized solution converges to a solution of Filippov's differential 
inclusion (\cite[Theorem 1, \S 8]{Filippov}). 
It follows that, if $\Sigma$ is a codimension $1$ discontinuity surface, 
for any regularization that satisfies the assumptions of  
\cite[Theorem $1$, $\S 8$]{Filippov}, the corresponding solution will converge to the unique sliding 
Filippov solution on $\Sigma$ as the regularization parameter goes to zero. 
But if $\Sigma$ has codimension $k \geq 2$, then the ambiguity of Filippov's
selection will reflect also in the amibiguity of the limit of regularized solutions
(in other words, the limit will depend on the chosen regularization).

Our goal in this section is to study the limiting behavior of  the solutions of 
a certain regularization, namely the bilinear interpolant, or simply bilinear,
regularization \eqref{bilinear_reg} below.
This regularization (or a close relative)
has often been employed and studied in the literature;
see \cite{AlexSeid1, AlexSeid2, Seid1, Seid2, DiGu, Jeffrey1, GuHairer, Llibre.Silva.Teixeira}.  
More specifically, we will study when the solution
of the bilinear regularization converges to the solution of a particular selection of 
the Filippov's vector field: the sliding bilinear vector field \eqref{bilinear}. 
When $\Sigma$ is nodally attractive,  
this convergence has been shown in \cite[Theorem 5.1]{AlexSeid1}, and -under the same assumption-
in \cite{DiGu} it is shown that the bilinear regularized vector field converges to the
bilinear sliding vector field. 
In what follows, we relax the hypothesis on attractivity of $\Sigma$, and simply
assume that $\Sigma$ is
\emph{attractive in finite time}, i.e. all trajectories with initial conditions in a neighborhood
of $\Sigma$ will reach $\Sigma$ in finite time.  This can be achieved  
either  upon sliding along $\Sigma_1$ and/or $\Sigma_2$ ($\Sigma$ is attractive upon sliding),
or spiralling around $\Sigma$ ($\Sigma$ is spirally attractive); 
see Definition \ref{AttrSigma}.
For later reference, and under the stated attractivity assumptions of $\Sigma$, we note that 
the algebraic system \eqref{bilinear} has a unique solution $(\alpha,\beta)$ in $(0,1)^2$.

For simplicity\footnote{a fact that can be always locally guaranteed},
we assume that $\Sigma$ is the intersection of the hyperplanes 
$\Sigma_1=\{x \in \R^n | x_1=0\}$ and  $\Sigma_2=\{x \in \R^n | x_2=0\}$. 
Then, we consider the $\epsilon$-neighborhood $S$ of $\Sigma$:
$S=\{x_1,x_2:\ -\epsilon\le x_1,x_2\le \epsilon\}$, 
and two smooth (at least $C^1$) monotone functions (the functions of course
depend on $\epsilon$, but for notational simplicity this 
dependence on $\epsilon$ is omitted)
$\alpha,\beta\ :\ \R \to [-1,1]$ interpolating at
$\pm \epsilon$, as follows:
\begin{equation}\label{conditions}
\alpha(\epsilon)=\beta(\epsilon)=1, \,\  
\alpha(-\epsilon)=\beta(-\epsilon)=0, \,\ 
\alpha'(x_1),\beta'(x_2)>0, \text{ for }-\epsilon< x_1,x_2<\epsilon\ . 
\end{equation}
We call \emph{bilinear regularization} the following one parameter family of
vector fields
\begin{equation}\label{bilinear_reg}\begin{split}
f_{\mathcal B}^{\epsilon}(x)=(1-\alpha(x_1))\left[
(1-\beta(x_2))f_1(x)+\beta(x_2)f_2(x)\right]+ \\
\alpha(x_1)(\left[(1-\beta(x_2))f_3(x)+\beta(x_2)f_4(x)\right]\ .
\end{split}
\end{equation} 
To be specific, in what follows, we have taken
\begin{equation}\label{alp_bet_eq}
\alpha(x_1)=\left \{ \begin{matrix} 1 & x_1>\epsilon \\ \frac 12 +\frac{x_1}{4 \epsilon} 
(3-(\frac{x_1}\epsilon)^2) & x_1 \in [-\epsilon, \epsilon] \\ 
0 & x_1<-\epsilon \end{matrix} \right . , \,\
\beta(x_2)=\left \{ \begin{matrix} 1 & x_2>\epsilon \\ \frac 12 +\frac{x_2}{4 \epsilon} 
(3-(\frac{x_2}\epsilon)^2) & x_2 \in [-\epsilon, \epsilon], \\ 
0 & x_2<-\epsilon \end{matrix} \right .
\end{equation}
but other choices of a $C^1$ monotone interpolant could be considered. 
Clearly, a choice of two different 
parameters for $\alpha$ and $\beta$, respectively $\epsilon_{\alpha}$ and $\epsilon_{\beta}$, 
could also be considered (e.g., see \cite{Seid1}),
and this would be justified if, for example, it is known a priori that the trajectories 
of un underlying physical system approach the two surfaces $\Sigma_1$ and $\Sigma_2$ at different rates.
Naturally, the choice of different parameters might lead to different qualitative behavior
of the corresponding solutions, as we will see in Section \ref{Numerics_section}. 

\begin{rem}\label{C0-int}
Often (e.g., see \cite{AlexSeid1, DiGu, Llibre.Silva.Teixeira}), a simpler $C^0$ regularization
is adopted, namely $\alpha(x_1)=\frac{x_1+\epsilon}{2\epsilon}$, and similarly
for $\beta(x_2)$, possibly with different $\epsilon_\alpha$ and $\epsilon_\beta$.
Proposition \ref{Qinvariant_prop} is not meaningful in this case, 
whereas Proposition \ref{as_stable_prop}
holds with essentially the same proof.  See also Remark \ref{RemDummy}.
\end{rem}

Our goal is to study the behavior of solutions of 
$\dot x=f^\epsilon_{\mathcal B}(x)$ for $x\in S$, and to see when, for $\epsilon \to 0$, 
they converge to the sliding solution on $\Sigma$ with vector field $f_B$ as given in \eqref{bilinear}. 
In order to do so, we  
split $\dot x=f^\epsilon_{\mathcal B}(x)$ into fast and slow motions.  
For 
$(x_1,x_2)$ in $S$,  we can rewrite \eqref{bilinear_reg} with respect to the variables 
$(\alpha, \beta, x_3, \ldots, x_n)$.  For the sake of more compact notation,
we let $y=(x_3, \ldots, x_n)$ and further $\dot y=g(\alpha,\beta,y)$.
In these variables, the full system rewrites as
\begin{equation}\label{full_system}\begin{split}
\dot \alpha & = \frac{d \alpha} {d x_1}
e_1^Tf^\epsilon_{\mathcal B}(\alpha, \beta, y):= \frac{d \alpha} {d x_1} g_1(\alpha, \beta, y) \\
\dot \beta & = \frac{d \beta} {d x_2}
e_2^Tf^\epsilon_{\mathcal B}(\alpha, \beta, y):= \frac{d \beta} {d x_2} g_2(\alpha, \beta, y), \\
\dot y & = g(\alpha, \beta, y) \qquad \quad 
\end{split}
\end{equation}
where $e_i$ is the standard $i$-th unit vector in $\R^n$, $i=1,2$.
Notice that $\frac{d \alpha}{d x_1}$ and 
$\frac{d \beta}{d x_2}$ depend on 
$\epsilon$ and are strictly positive (inside $S$); from \eqref{alp_bet_eq}, they are equal to 
$\frac{3}{4 \epsilon}(1-\frac{x_i^2}{\epsilon^2})$, with $i=1,2$.
We refer to $\alpha$ and $\beta$ as the fast variables and to $y$ as the slow variable. 
We denote the solution 
of \eqref{full_system} as $(\alpha_{\epsilon}(\cdot),\beta_{\epsilon}(\cdot),y_{\epsilon}(\cdot))$. 

Now, using \eqref{alp_bet_eq}, in \eqref{full_system}, because of monotonicity of $\alpha$ and $\beta$,
we can rewrite $x_1$ as a function of $\alpha$ and $x_2$ as a function of $\beta$. 
From \eqref{alp_bet_eq}, let $z=\frac{x_1}{\epsilon}$ and rewrite the first of \eqref{alp_bet_eq},
for $x_1\in [-\epsilon,\epsilon]$, as: $z^3-3z+4\alpha-2=0$. Using Vieta's substitution 
$z=w+\frac 1w$, we get the equation $w^6+(4\alpha-2)w^3+1=0$.  Then 
$w^3=(1-2\alpha) \pm 2 i \sqrt{\alpha-\alpha^2}$ and its third roots have modulus one, 
so that $z=w+\frac 1w=\bar w+\frac 1{\bar w}$ is real.
Let $\varphi(\alpha)=Arg \left ((1-2\alpha)+2i\sqrt{\alpha-\alpha^2} \right )$, then 
for $\alpha$ in [0,1], $\varphi(\alpha)$ is in $[0,\pi]$, and for 
$w(\alpha)=e^{i (\frac{\varphi(\alpha)}3+\frac 43 \pi})$, 
the corresponding $z(\alpha)$ can be rewritten as  
$z(\alpha)=2\cos(\frac{\varphi(\alpha)}3 +\frac 43 \pi)$ 
and satisfies $z(0)=-1$ and $z(1)=1$, this $z(\alpha)$ is the function we are looking for. 
Same reasoning applies for $\beta$. 
Then system \eqref{full_system} rewrites as
\begin{equation}\label{full_system_2}\begin{split}
\epsilon \dot \alpha & =\frac {3}{4} 
\left (1-4\cos^2 \left ( \frac{\varphi (\alpha)}{3}+\frac 43 \pi \right ) \right ) 
e_1^Tf_{\mathcal B}(\alpha, \beta, y) \\
 \epsilon \dot \beta & =\frac{3}{4} 
 \left (1-4\cos^2 \left (\frac{\varphi (\beta)}3 +\frac 43 \pi \right ) \right )
e_2^Tf_{\mathcal B}(\alpha, \beta, y), \\
\dot y & =g(\alpha,\beta,y) 
\end{split}
\end{equation}
with $\varphi(\gamma)=Arg \left ((1-2\gamma)+2i\sqrt{\gamma -\gamma^2} \right )$.
Notice that the function 
$\left (1-4 \cos^2 \left ( \frac{\varphi (\gamma)}{3}+\frac 43 \pi \right ) \right )$ 
is strictly positive for $\gamma \in (0,1)$ and it is $0$ at $\gamma=0,1$. 
Following standard approaches for singularly perturbed systems, we set 
$\epsilon=0$ in \eqref{full_system_2} and 
obtain the following system
\begin{equation}\label{slow_system}\begin{split}
0 & = g_1(\alpha,\beta,y) \\
0 & =g_2(\alpha,\beta,y) \\
\dot y & =g(\alpha,\beta, y)
\end{split}
\end{equation}
Notice that solutions of \eqref{slow_system} are sliding solutions on $\Sigma$ 
with bilinear vector field $f_B$ as in \eqref{bilinear}. 
Let $(\alpha^*(y),\beta^*(y))$ denote the solution of  
\begin{equation}\label{algebraic_eq}
\begin{cases}g_1(\alpha^*(y), \beta^*(y), y)& =0 \\ g_2(\alpha^*(y), 
\beta^*(y), y))&=0\end{cases}.
\end{equation}
(Recall that, under the assumptions 
of attractivity by sliding or by spiralling, 
$(\alpha^*(y),\beta^*(y))$ is the unique solution of \eqref{algebraic_eq} in $[0,1]^2$). 

Our goal in this section is twofold: to see if solutions of \eqref{full_system} converge 
to solutions of \eqref{slow_system} as $\epsilon \to 0$ and to explore the behavior of 
solutions of \eqref{full_system} in the neighborhood of generic exit points. 

\subsection{Asymptotic behavior of the regularized problem for $\Sigma$ attractive in finite time} 
From \eqref{full_system}, we
introduce the time variable $\tau=\frac {3t}{4\epsilon}$ and consider the fast system
\begin{equation}\label{fast_system}\begin{split}
\alpha'=\left (1-4\cos^2 \left ( \frac{\varphi (\alpha)}{3}+\frac 43 \pi \right ) \right )
g_1(\alpha,\beta,y) \\
\beta'=\left (1-4\cos^2 \left ( \frac{\varphi (\beta)}{3}+\frac 43 \pi \right ) \right )
g_2(\alpha,\beta,y), 
\end{split}
\end{equation}
where the ``prime'' denotes differentiation with respect to $\tau$, and $y$ is considered as
a vector of parameters. Notice that the solution $(\alpha^*(y),\beta^*(y))$ of \eqref{algebraic_eq} is 
an equilibrium of \eqref{fast_system}.  
We denote with $(\alpha(t,y),\beta(t,y))$ the solution of \eqref{fast_system} with 
initial condition $(\alpha_0,\beta_0)$.

\begin{rem}
In case we take different $\epsilon_\alpha$ and $\epsilon_\beta$ in \eqref{alp_bet_eq},
we can write $\epsilon_\beta=\eta \epsilon_\alpha$, perform the change of
time variable $\tau=\frac {3t}{4\epsilon_\alpha}$, and obtain (similarly to 
\eqref{fast_system}) the fast system
\begin{equation}\label{fast_system2}\begin{split}
\alpha'=\left (1-4\cos^2 \left ( \frac{\varphi (\alpha)}{3}+\frac 43 \pi \right ) \right )
g_1(\alpha,\beta,y) \\
\beta'=\frac{1}{\eta}
\left (1-4\cos^2 \left ( \frac{\varphi (\beta)}{3}+\frac 43 \pi \right ) \right )
g_2(\alpha,\beta,y)\ . 
\end{split}
\end{equation}
\end{rem}

The following result by Artstein is a stronger version of a classical result in singular 
perturbation theory (\cite[Theorem 2.1]{Artstein}).

\begin{theorem}\label{ClassicSingPert}
Assume that 
\begin{itemize}
\item[(i)] the solution $(\alpha^*(y),\beta^*(y))$ of \eqref{algebraic_eq}
is continuous in $y$; 
\item[(ii)] $(\alpha^*(y),\beta^*(y))$ is a locally asymptotically stable equilibrium 
of the fast system \eqref{fast_system}; 
\item[(iii)] the initial condition $(\alpha_0,\beta_0,y_0)$ of \eqref{full_system} 
is such that the $\omega$-limit set of $(\alpha(t,y_0),\beta(t,y_0))$, is $(\alpha^*(y_0),\beta^*(y_0))$;
\item[(iv)] 
the problem $\dot y=g(\alpha^*(y),\beta^*(y),y)$, $y(0)=y_0$, has a unique solution, 
denote it as $y_0(t)$.
\end{itemize}
Then, the solution of \eqref{full_system} with initial condition $(\alpha_0,\beta_0,y_0)$, is such that 
as $\epsilon \to 0$:
\begin{itemize}
\item [a)]$y(\cdot)$ converges to $y_0(\cdot)$, uniformly in time on intervals of the form $[0,T]$;
\item [b)] $(\alpha(t),\beta(t))$ converges to 
$(\alpha^*(y_0(t)),\beta^*(y_0(t)))$ uniformly in time on intervals of the form $[\delta, T]$, $\delta>0$. 
\end{itemize}
\qed
\end{theorem}

Therefore, when conditions (i)-(iv) of Theorem \ref{ClassicSingPert}
are verified, the solution of the regular system converges 
to the sliding solution with bilinear vector field \eqref{bilinear}.  

\begin{rem}\label{RemDummy}
If we used (see Remark \ref{C0-int})
the $C^0$ regularization $\alpha(x_1)=\frac{x_1+\epsilon}{2\epsilon}$, 
$\beta(x_2)=\frac{x_2+\epsilon}{2\epsilon}$, and let $\tau=\frac{t}{2\epsilon}$, then
we would have obtained the system
\begin{equation}\label{dummy}\begin{split}
\alpha'=g_1(\alpha,\beta,y) \\
\beta'= g_2(\alpha,\beta,y)
\end{split}
\end{equation}
instead of \eqref{fast_system}.  Now, \eqref{dummy} is precisely the ``dummy'' system of 
\cite{Jeffrey1} and \cite{GuHairer}.  However, 
\eqref{fast_system} is not orbitally equivalent to \eqref{dummy}. As we will see 
in Proposition \ref{as_stable_prop}, under appropriate conditions of attractivity of $\Sigma$, 
$(\alpha^*(y),\beta^*(y))$ is an asymptotically stable equilibrium for both
\eqref{fast_system} and \eqref{dummy}. However, condition $(iii)$ in Theorem \ref{ClassicSingPert}, 
implies that the limiting behavior of the solution of \eqref{full_system}, depends also on 
the basin of attraction of $(\alpha^*(y),\beta^*(y))$ and this, in general, 
is not the same for the two systems.
As an illustration of this, see Example \ref{example_2} in Section \ref{Numerics_section}.  
System \eqref{fast_system} is the system we need 
to study in order to understand the limiting behavior of the regularized solution in the case
of the $C^1$ regularization \eqref{alp_bet_eq}.  Moreover, we note that if we had used the
$C^0$ regularization with different parameters $\epsilon_\alpha$ and $\epsilon_\beta$, namely
$\alpha(x_1)=\frac{x_1+\epsilon_\alpha}{2\epsilon_\alpha}$ and
$\beta(x_2)=\frac{x_2+\epsilon_\beta}{2\epsilon_\beta}$, then we would obtain a system not
orbitally equivalent to \eqref{dummy}.
\end{rem}

In what follows, we first assume that $\Sigma$ is attractive in finite time 
(upon sliding or spirally) and we want to verify if and when (i), (ii) and (iii) are satisfied. 
This in turn will imply that solutions of the regularized problem converge to the sliding solution 
on $\Sigma$ with vector field $f_B$ as given in  \eqref{bilinear}.   
The following hold, both when $\Sigma$ is attractive upon sliding or spirally.
\begin{itemize}
\item[(a)]  System \eqref{fast_system} has a unique equilibrium 
$(\alpha^*(y),\beta^*(y))$ in $(0,1)^2$. 
\item[(b)] The functions $\alpha^*(y),\beta^*(y)$ are smooth functions of $y$. 
(For $\Sigma$ attractive upon sliding, this is in \cite[Theorem 8]{DiElLo}, and
the case of $\Sigma$ spirally attractive is analogous.)
Notice that this point (b) implies (iv) of Theorem \ref{ClassicSingPert}. 
\end{itemize}

\begin{rem}\label{unique_eq_remark}
We emphasize that the assumption that $\Sigma$ is attractive in finite time is 
sufficient but not necessary for the uniqueness of $(\alpha^*(y),\beta^*(y))$; e.g.,
see \cite{DiElLo}.
As we will see in Section \ref{Numerics_section}, this in turn will impact the behavior 
of the regularized solution, that might remain close to $\Sigma$ even when 
$\Sigma$ is not attractive.  
\end{rem}

In Proposition \ref{Qinvariant_prop} and Proposition \ref{as_stable_prop}, 
we study the dynamics of \eqref{fast_system} to verify asymptotic stability of $(\alpha^*,\beta^*)$. 

\begin{proposition}\label{Qinvariant_prop}
The phase space of \eqref{fast_system} is the square $Q=[0,1]^2$. 
Moreover, the boundary of $Q$, denoted as  $\partial Q$, is an invariant set for all 
values of $y$. The vertices of $Q$ are equilibria. If any of the sliding vector fields 
$f_{\Sigma_{1,2}}^\pm$ is well defined (i.e., if there is sliding on any of $\Sigma_{1,2}^\pm$), 
the corresponding value of $(\alpha(\bar y),\beta(\bar y))$ in \eqref{bilinear} is an 
equilibrium of \eqref{fast_system} for $y=\bar y$.   
\end{proposition}
\begin{proof}
For $\alpha=0,1$, or $\beta=0,1$, it must be  $x_1=-\epsilon,\epsilon$ 
or $x_2=-\epsilon, \epsilon$, and hence $\alpha'=0$ or
$\beta'=0$.  
Then, for all $y$,  the boundary of $Q$ is invariant under the flow of \eqref{fast_system}. 
It also follows that the vertices of $Q$ are equilibria of \eqref{fast_system}. 
Notice that the corresponding values of $g_1$ and $g_2$ at the equilibria are 
$(w^1_1,w^2_1)$ for $(0,0)$, $(w^1_3, w^2_3)$ for $(1,0)$, $(w^1_2,w^2_2)$ for $(0,1)$, 
and $(w^1_4,w^2_4)$ for $(1,1)$, where the $w^i_j$'s are defined in \eqref{bilinear}. 

\noindent In addition to the vertices of $Q$, other equilibria of 
\eqref{fast_system} might belong to $\partial Q$, as follows.
\begin{itemize}
\item If, for $y=\bar y$, there is sliding on ${\Sigma_2^-}$, then there exists $\beta^-(\bar y)$ such that 
$g_2(0,\beta^-(\bar y),\bar y)=0$ and  $(0,\beta^-(\bar y))$ is an equilibrium of \eqref{fast_system}. 
The vector field \eqref{bilinear} for $(\alpha,\beta)=(0,\beta^-(\bar y))$ is $f_{\Sigma_2^-}(\bar y)$.
\item If, for $y=\bar y$, there is sliding on ${\Sigma_2^+}$ then  there exists $\beta^+(\bar y)$ such that 
$g_2(1,\beta^+(\bar y),\bar y)=0$ and $(1,\beta^+(\bar y))$ is an equilibrium of \eqref{fast_system}.
The vector field \eqref{bilinear} for $(\alpha,\beta)=(0,\beta^+(\bar y))$ is $f_{\Sigma_2^+}(\bar y)$.
\item If, for $y=\bar y$, there is sliding on ${\Sigma_1^-}$ for $y=\bar y$, then  
there exists $\alpha^-(\bar y)$ 
such that $g_1(\alpha^-(\bar y),0,\bar y)=0$ and $(\alpha^-(\bar y),0)$ is an equilibrium 
of \eqref{fast_system}.
The vector field \eqref{bilinear} for $(\alpha,\beta)=(\alpha^-(\bar y),0)$ is $f_{\Sigma_1^-}(\bar y)$.
\item If, for $y=\bar y$, there is sliding on ${\Sigma_1^+}$ for $y=\bar y$, then 
there exists $\alpha^+(\bar y)$ 
such that $g_2(\alpha^+(\bar y),1,\bar y)=0$ and $(\alpha^+(\bar y),1)$ is an equilibrium 
of \eqref{fast_system}.
The vector field \eqref{bilinear} for $(\alpha,\beta)=(\alpha^+(\bar y),1)$ is $f_{\Sigma_1^+}(\bar y)$.
\end{itemize} 
\end{proof}

In Proposition \ref{as_stable_prop} we give sufficient conditions for  
$(\alpha^*,\beta^*)$ to be locally asymptotically stable.

\begin{proposition}\label{as_stable_prop}
Assume that $\Sigma$ is attractive in finite time upon sliding along 
at least two of the $\Sigma_{1,2}^{\pm}$, 
and that there is attractive sliding along these codimension 1 surfaces. 
Then $(\alpha^*(y),\beta^*(y))$ is exponentially asymptotically stable for \eqref{fast_system}.  
\end{proposition}
\begin{proof}
We prove the result for a particular configuration of vector fields in a neighborhood of $\Sigma$. 
The proof for all the other configurations is analogous. 
We consider the case in Figure \ref{CaseS1pS2p1_fig}, where $\Sigma$ is attractive upon sliding
along $\Sigma_{1,2}^+$, as characterized by the signs of the $w^i_j$'s 
in Table \ref{CaseS1pS2p1_table}, where the $w^i_j$'s are defined in \eqref{Wij}. 
\begin{table}\caption{}\label{CaseS1pS2p1_table}
\begin{tabular}{|c|c|c|c|c|}
\hline
 & $i=1$ & $i=2$ & $i=3$ & $i=4$ \\
 \hline
$w^1_i$ & $+$ & $+$ & $+$ & $-$ \\
\hline
$w^2_i$ & $-$ & $-$ & $+$ & $-$ \\ 
\hline
\end{tabular}
\end{table}   

\begin{figure}\caption{Proposition \ref{as_stable_prop}. 
Reference configuration under which $\Sigma$ is attractive in finite time.}
\label{CaseS1pS2p1_fig}
\includegraphics[width=100pt]{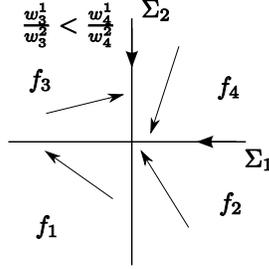}
\end{figure}

$\Sigma$ attractive implies that \eqref{full_system} has a unique equilibrium 
$(\alpha^*,\beta^*) \in (0,1)^2$.  To study the stability of $(\alpha^*,\beta^*)$ 
we consider the Jacobian matrix of \eqref{fast_system}, evaluated at $(\alpha^*,\beta^*)$,
that is we look at the eigenvalues of
\begin{equation}\label{jacobian_mat}\begin{split}
J&=\begin{pmatrix} \left ( 1-4\cos^2(\frac{\varphi(\alpha)}3+\frac 43 \pi) \right )
\frac{\partial g_1}{\partial \alpha}(\alpha^*,\beta^*) & 
\left ( 1-4\cos^2(\frac{\varphi(\alpha)}3+\frac 43 \pi) \right ) 
\frac{\partial g_1}{\partial \beta}(\alpha^*,\beta^*) \\
\left ( 1-4\cos^2(\frac{\varphi(\beta)}3+\frac 43 \pi) \right ) 
\frac{\partial g_2}{\partial \alpha} (\alpha^*,\beta^*) & 
\left ( 1-4\cos^2(\frac{\varphi(\beta)}3+\frac 43 \pi) \right )
\frac{\partial g_2}{\partial \beta}(\alpha^*,\beta^*) 
\end{pmatrix}=\\
&\begin{pmatrix} 1-4\cos^2(\frac{\varphi(\alpha)}3+\frac 43 \pi) & 
0  \\ 0  & 
1-4\cos^2(\frac{\varphi(\beta)}3+\frac 43 \pi) 
\end{pmatrix}\, \tilde J\ , \,\
\tilde J:=\begin{pmatrix} 
\frac{\partial g_1}{\partial \alpha}(\alpha^*,\beta^*) &  
\frac{\partial g_1}{\partial \beta}(\alpha^*,\beta^*) \\
\frac{\partial g_2}{\partial \alpha} (\alpha^*,\beta^*) & 
\frac{\partial g_2}{\partial \beta}(\alpha^*,\beta^*) 
\end{pmatrix}\ .
\end{split}
\end{equation}
Then, as in the proof of \cite[Theorem 8]{DiElLo}, it follows that $\text{det}(\tilde J)>0$, and hence 
$\text{det}(J)>0$. 
We now show that $\text{trace}(J)$ is negative. We write explicitly the diagonal elements of $\tilde J$ 
\begin{equation*}
\begin{split}
\frac{\partial g_1}{\partial \alpha}(\alpha^*,\beta^*)=
-(1-\beta^*)w^1_1-\beta^* w^1_2+(1-\beta^*)w^1_3+\beta^* w^1_4 \\
\frac{\partial g_2}{\partial \beta}(\alpha^*,\beta^*)=
-(1-\alpha^*)w^2_1+(1-\alpha^*)w^2_2-\alpha^* w^2_3+\alpha^* w^2_4.
\end{split}
\end{equation*}
and the equilibrium 
$(\alpha^*,\beta^*)=\left (\frac{C_1(\beta^*)}{C_1(\beta^*)-C_2(\beta^*)},
\frac{L_1(\alpha^*)}{L_1(\alpha^*)-L_2(\alpha^*)} \right )$,
where $C_1(\beta^*)=(1-\beta^*)w^1_1+\beta^*w^1_2$, $C_2(\beta^*)=(1-\beta^*)w^1_3+\beta^*w^1_4$, 
$L_1(\alpha^*)=(1-\alpha^*)w^2_1+\alpha^*w^2_3$, $L_2(\alpha^*)=(1-\alpha^*)w^2_2+\alpha^*w^2_4$. 
Notice that $0 \leq \alpha^* \leq 1$ together with $C_1(\beta^*)>0$ (see Table \ref{CaseS1pS2p1_table}) 
implies that $\frac{\partial g_1}{\partial \alpha}(\alpha^*,\beta^*)<0$. 
Similarly, $0\le \beta^* \leq 1$ together with 
$L_2(\alpha^*)<0$ implies $\frac{\partial g_2}{\partial \beta}(\alpha^*,\beta^*)<0$. This together with 
$\left ( 1-4\cos^2(\frac{\varphi(\gamma)}3+\frac 43 \pi) \right )>0$, gives the sought result.
\end{proof}

The assumptions in Propositions \ref{as_stable_prop} 
exclude the case of $\Sigma$ 
spirally attractive and $\Sigma$ attractive upon sliding on just one of the $\Sigma_{1,2}^\pm$. 
For each of these cases, examples can be given where the equilibrium $(\alpha^*,\beta^*)$ 
of \eqref{fast_system} is unstable even if $\Sigma$ is attractive.
See Example \ref{example_3}
in Section \ref{Numerics_section}, where, with $\epsilon_{\beta}=10\epsilon_{\alpha}$, the 
equilibrium $(\alpha^*,\beta^*)$ of the fast system is unstable while $\Sigma$ is attractive.  

\subsection{Behavior of the regularized solution in the neighborhood of exit points}
\label{Exitpoints_section}
Here, we are interested in studying how 
solutions of \eqref{full_system} behave when $\Sigma$ looses attractivity.
Assume first that $\Sigma$ is attractive upon sliding. We will focus on the case when 
$\Sigma$ looses attractivity at so-called {\emph{potential tangential exit points}} and at
{\emph{potential non-tangential exit points}}, defined next. \\

\begin{defn}\label{PotExitPts}
Let $\bar x$ be a point on $\Sigma$. We say that $\bar x$ is a {\emph{potential tangential exit point}}
if, at $\bar x$, one of the vector fields $f_{\Sigma_{1,2}}^\pm$ is tangent to $\Sigma$ 
(and hence it belongs to the Filippov convex combination). 
We say that $\bar x$ is a {\emph{potential non-tangential}} exit point if, at $\bar x$, one of 
the vector fields $f_j(x)$ is tangent either to $\Sigma_1$ or to $\Sigma_2$ and it points 
away from $\Sigma$.  
\end{defn}

{\bf Simplification: $\Sigma$ is a curve}.
For simplicity, below we restrict to the case $n=3$, so that, in our case of
$\Sigma=\{x_1=x_2=0\}$, $\Sigma$ is a (portion of the) $x_3$-axis. 
(Of course, 
in general, when the co-dimension 2 discontinuity manifold $\Sigma$ is immersed in $\R^n$, with $n\ge 3$,
we expect that exit points themselves will lie on a (union of disjoint) $(n-3)$-dimensional
manifold(s) of codimension 3.  With the understanding that exit points are now lying on these
manifolds, our description below still qualitatively holds). Also, 
we consider only first order phenomena i.e., only potential tangential (respectively,
non-tangential) exit points $\bar x$, such that the derivative with respect to time
(evaluated along the trajectory) of $f_{\Sigma_{1,2}}^\pm(\bar x)$ 
(respectively, $f_j(\bar x)$, $j=1,\dots, 4$) 
is different from zero.  As a consequence, for us exit points are isolated points on $\Sigma$. 
In this case of $n=3$, and insofar as the behavior of the regularized solution in
the neighborhood of tangential and non-tangential exit points, the following phenomena can arise.

\subsubsection{Tangential exit points}  What occurs in this
case depends on the number of equilibria of \eqref{fast_system}.

Suppose that the regularized solution enters the neighborhood of a tangential exit point 
$\bar x=(0,0, \bar x_3)$, and without loss of generality let $f_{\Sigma_1^+}(\bar x)$ be the
vector field tangent to $\Sigma$, and let $\Sigma$ be attractive for $x_3< \bar x_3$.
Then either $(a)$ or $(b)$ below can happen.
\begin{itemize}
\item[(a)] $(\alpha^*(\bar x_3),\beta^*(\bar x_3))=(\alpha^+(\bar x_3),1)$ and  the regularized 
solution exits the neighborhood of 
$\Sigma$ to enter a neighborhood of $\Sigma_1^+$;
\item[(b)] $(\alpha^*(\bar x_3),\beta^*(\bar x_3)) \neq (\alpha^+(\bar x_3),1)$  
and these are the only solutions of 
\eqref{algebraic_eq} for $y=x_3=\bar x_3$; a first order analysis guarantees that the 
solution that at $\bar x_3$ is equal to $(\alpha^+(\bar x_3),1)$ is in the interior of 
$[0,1]^2$ for $x_3$ in a right neighborhood of $\bar x_3$.  
Nonetheless, $(\alpha^*(x_3(t)),\beta^*(x_3(t)))$ retains its asymptotic stability. 
It follows that the regularized 
solution remains close to $\Sigma$ and for $\epsilon \to 0$ it will converge to the solution of 
the bilinear vector field on $\Sigma$, still well defined even though $\Sigma$ is not 
locally attractive anymore. Eventually 
the regularized solution $x(t)$ will leave a neighborhood of 
$\Sigma$ if one of the two following phenomena takes place:
\begin{itemize}
\item[(i)] the equilibrium of \eqref{fast_system} reaches a fold bifurcation value; 
\item[(ii)] $(\alpha^*(x(t)),\beta^*(x(t)))$ is on the boundary of $Q$ for $t=\hat t$,
and this can happen only if one of the vector fields $f_{\Sigma_{1,2}}^\pm(x(\hat t))$ 
is tangent to $\Sigma$, say $f_{\Sigma_2}^+$. 
\end{itemize}
In case (ii), the regularized solution will leave the neighborhood of 
$\Sigma$ to enter a neighborhood of $\Sigma_2^+$. 
But, in case (i) it is not generally possible to predict how the regularized solution
will leave a neighborhood of $\Sigma$.
\end{itemize}
The phenomena (a) and (b) above do not depend on $\frac{d\alpha}{dx_1}$ and 
$\frac{d\beta}{dx_2}$ in \eqref{fast_system}, but only on the solution of the
algebraic part in \eqref{slow_system}. Hence any choice of functions
$\alpha$ and $\beta$ in \eqref{bilinear_reg} that satisfies \eqref{conditions}
will lead to the same phenomena. 

\subsubsection{Non tangential exit points}
Suppose that the regularized solution enters the neighborhood of a non-tangential 
exit point, $\bar x=(0,0,\bar x_3)$.  
In this case, although there is only one solution of \eqref{slow_system}, different
things can happen depending on the functions $\frac{d\alpha}{dx_1}$ and 
$\frac{d\beta}{dx_2}$ in \eqref{fast_system}.

Suppose that $f_1$ is the vector field verifying the 
conditions for non tangential exit.  
Then $(\alpha^*(\bar x_3),\beta^*(\bar x_3))$ is still the only equilibrium of \eqref{fast_system} 
in the interior of $[0,1]^2$ and it is asymptotically stable. However, the equilibrium 
$(0,0)$ undergoes a bifurcation at $x_3=\bar x_3$ and, for $x_3>\bar x_3$, 
there is a neighborhood of $(0,0)$ in $Q$ such that all solutions in that 
neighborhood are attracted to $(0,0)$.  

As far as the regularized solution, (i) or (ii) below may occur.

\begin{itemize}
\item[(i)]
The regularized solution $x(t)$ might remain close to $\Sigma$, in agreement with
Theorem \ref{ClassicSingPert}.
\item[(ii)] The regularized solution will leave a neighborhood of $\Sigma$ to enter $R_1$, 
if the corresponding solution of \eqref{fast_system} enters the neighborhood of solutions 
that reach $(0,0)$. 
\end{itemize}
In other words, the behavior of the regularized solution depends on the
possible bifurcations of $(\alpha^*,\beta^*)$ and of $(0,0)$.  Moreover,
we can choose the functions $\alpha$ and $\beta$ so that the terms 
$\frac{d\alpha}{dx_1}$ and $\frac{d\beta}{dx_2}$ will make
either one of (i) or (ii) above take place. 
See Example \ref{example_2} in Section \ref{Numerics_section} for an 
illustration of this fact.

\subsubsection{Exit point in the case of spiral dynamics}
We consider the case of $\Sigma$ spirally attractive and, based on the results in 
\cite{DieciSpiral}, we propose the following definition of potential spiral exit point. 
\begin{defn}\label{ExitPtSpiral}
Let $\bar x \in \Sigma$ and assume that there is spiral dynamics around $\Sigma$. 
We say that $\bar x$ is a potential spiral exit point if 
$\bar x$ is such that
$$ \frac{w^2_1(\bar x)w^1_3(\bar x)w^2_4(\bar x)w^1_2(\bar x)}
{w^1_1(\bar x)w^2_3(\bar x)w^1_4(\bar x)w^2_2(\bar x)}=1\ .$$
\end{defn}    
Again, we are only concerned with first order phenomena, i.e., phenomena such that the 
derivative with respect
to time of the quantity on the left hand side in Definition \ref{ExitPtSpiral} evaluated at $x(t)=\bar x$, 
is different from zero.

The equilibrium $(\alpha^*,\beta^*)$ of \eqref{fast_system} is always 
unique and well defined in $(0,1)^2$ for all $x_3 \in \R$ as long as there is spiral 
dynamics around $\Sigma$. The attractivity of $\Sigma$ though, does not imply the stability 
of the equilibrium. On the other hand, $(\alpha^*,\beta^*)$ might be stable when $\Sigma$ 
is not attractive. As an illustration of both phenomena, see Example \ref{example_3}. 
Furthermore, when the equilibrium of \eqref{fast_system} is unstable and the trajectories 
of \eqref{fast_system}  reach (possibly in finite time) one of the equilibria on the boundary 
of $Q$, it is the local dynamics around $\Sigma$ that determines the behavior of the 
regularized solution. This is well illustrated in Example \ref{example_3}, Figure 
\ref{spiral_ode23s_fig}, for $\epsilon_{\beta}=10\epsilon{\alpha}$.           
        
Based upon the above described situation insofar as different behaviors of the
regularized solution, it is natural to ask how 
should the solution of the original discontinuous system \eqref{GeneralPWS} behave once a 
potential first order exit point is reached, and how we should infer this.
We discuss this aspect in the next section.

\section{Numerical simulations for the discontinuous system}\label{Euler_section}

Frequently, Euler's method has been used as a
mean to approximate the behavior of  solutions of \eqref{GeneralPWS} in a 
neighborhood of the discontinuity surface. 
In \cite{Filippov} (proof of Theorem 1, page 77), Filippov showed that the solutions obtained 
with Euler's method converge to one of the solutions of Filippov's inclusion when the 
discretization stepsize goes to zero. 
In \cite{AlexSeid1}, Alexander and Seidman --in the case of a nodally attractive 
codimension 2 discontinuity surface $\Sigma$-- 
consider a chattering trajectory 
$x_{\epsilon}$ that evolves in an $\epsilon$-neighborhood of $\Sigma$. 
The trajectory $x_{\epsilon}$ is obtained by considering
the Euler approximation of the solution, and the fact that $\Sigma$ is nodally attractive
guarantees that --for a sufficiently small stepsize 
$\tau$-- the Euler's approximations remain in an $\epsilon$-neighborhood of $\Sigma$. 
Alexander and Seidman in \cite{AlexSeid1} show that \emph{every possible} Filippov
sliding vector field on $\Sigma$ is realizable;
i.e.,  given a solution $\bar x(t)$ of Filippov's differential inclusion, 
there exists a chattering trajectory $x_{\epsilon}$ that converges uniformly in $t$ to 
$\bar x$ in a given time interval. 

\subsection{Discretization}
We also use Euler's method as a tool to understand how the solutions of \eqref{GeneralPWS}
should behave in a neighborhood of the codimension $2$ 
discontinuity surface $\Sigma$.   However, we propose to do this in a new way, which
enables us to make (statistical) inferences on the {\emph{expected}} behavior of a trajectory.

First, we evolve by Euler's method several initial conditions in a neighborhood of $\Sigma$.
Second, in order to mimic the non ideal behavior of 
a physical system, at each step $k$, 
we use a random stepsize $\tau_k$ such that $\tau\leq \tau_k \leq 2 \tau$, where
$\tau$ is a fixed reference value. In case one of the computed approximations 
falls on a discontinuity surface\footnote{a case that has never occurred in the several
thousands experiments we have performed}, we take a random perturbation of size of machine 
precision, so that the perturbed value belongs to one of the regions $R_j$'s. 
Therefore, for each initial condition $x_0$, chosen in a $\tau$-neighborhood of $\Sigma$, 
we generate the following approximate solution
\begin{equation}\label{euler_sol}
x_{k+1}=\left \{ \begin{matrix} x_k+\tau_k f_j(x_k), & x_k \in R_j , \\ 
x_k+\delta_k +\tau_k f_j(x_k+\delta_k), & x_k \in \Sigma, \Sigma_{1,2}, \end{matrix} \right .
\end{equation} 
where $\delta_k>0$ is a random perturbation the same size of machine precision.
We perform two types of experiments with the scheme \eqref{euler_sol}.
\begin{itemize}
\item[(a)]
We generate 100 (or more) random initial conditions in a $\tau$-neighborhood of $\Sigma$, and evolve 
each of them according to \eqref{euler_sol} on a given time interval of interest.  
We further monitor, see below, exit points for each
trajectory, and perform statistics on these.
\item[(b)]
For a given (randomly generated) initial condition in a $\tau$-neighborhood of $\Sigma$, we
generate 100 trajectories according to \eqref{euler_sol}, on a given time interval of interest.    
Of course, becase of the randomnes in the stepsize selection process, these will give approximations
at different times.  Thus, we further interpolate linearly the given trajectories on a fixed temporal
grid with spacing $\tau$, that is at times $0,\tau, 2\tau, \dots\ $.  Finally, we compute an
{\emph{average trajectory}} by averaging the obtained 100 approximations on the fixed grid.
We also compute exit points, etc., with respect to this average trajectory.
\end{itemize}

\subsection{Construction of the experiments}\label{ConstrExp}
As previously remarked, our experiments are all in $\R^3$ or $\R^4$, and with
discontinuity surfaces given by $h_1(x)=x_1$ and $h_2(x)=x_2$.  Moreover, we choose the
vector fields so that:
\begin{itemize}
\item All but one of the $w^i_j$'s are constant and they are in absolute value less than $1$. 
The non-constant $w^i_j$ is a function of the slow variables $x_3, \ldots, x_n$ and it 
is chosen in such a way that $\Sigma$ changes from locally 
attractive in finite time to non  attractive; \\
\item No vector field in the Filippov differential inclusion has equilibria (or more
complicated invariant sets) on $\Sigma$.
\end{itemize}

\subsection{Exits}
An important aspect of our simulations will be to monitor if a Euler approximation
$\{x_k\}_{k \in \N}$ of \eqref{GeneralPWS}, computed as in \eqref{euler_sol},
leaves a neighborhood of $\Sigma$ when $\Sigma$
loses attractivity.  To perform this task, we reasoned as follows.
(The choices below are appropriate for the vector fields we chose, see Section \ref{ConstrExp}).
\begin{itemize}
\item[(a)] \emph{Exit on a codimension $1$ surface} [Tangential Exits]. 
Suppose that $\Sigma$ loses attractivity at a tangential exit point
for which $f_{\Sigma_1^+}$ becomes tangent to $\Sigma$.
We monitor the function $h_2$ and declare that the numerical solution 
leaves a neighborhood of $\Sigma$ by checking that $h_2(x_k)>10 \tau$.
When this is the case, we define as ``exit point'' the 
value $\bar x=\frac{x_{\bar k}+x_{\bar k+1}}{2}$, 
where the index $\bar k$ is the first index $\bar k$ for which
$h_2(x_{\bar k})>1.5 \tau$ together with 
$h_2(x_j)>\tau$ for all $j \geq \bar k$.  Note that we do this both for the
100 trajectories generated by \eqref{euler_sol} with different initial
condition, as well as for the average trajectory.
\item[(b)] \emph{Exit in one of the $R_j$'s} [Non Tangential Exits].
Suppose, for example, that the non-tangential exit is into the region $R_1$. 
We observe that, at the non-tangential exit point, $f_1$ is tangent either to 
$\Sigma_1$ or to $\Sigma_2$, while pointing away from $\Sigma$. 
This implies the existence of repulsive siding along $\Sigma_1$ or $\Sigma_2$  
with sliding vector field $f_1$ at the non-tangential exit point. 
Hence, we still need to detect first the exit on a codimension 1 sliding surface,
which we do as in (a).
\item[(c)] \emph{Exit from spiral case}.  As long as $\Sigma$ is spirally attractive,
motion around $\Sigma$ now repeatedly takes the trajectory inside all of the regions
$R_1,\dots, R_4$. To decide if the trajectory leaves a neighborhood of $\Sigma$ (when
the spiral motion ceases to be attractive), 
we monitor the 2-norm of the vector $(h_1,h_2)$ 
for the last computed numerical value in $R_4$, before the solution enters the 
neighboring region $R_2$ (clockwise motions around $\Sigma$) or $R_3$ 
(counterclockwise motion).   We do this 
for the trajectories associated to different initial
conditions, and declare an exit when the $2$-norm of $(h_1,h_2)$ becomes strictly monotone increasing. 
\end{itemize}

\subsection{Summary and limitations}
As we report in the next section, 
based upon our experiments with Euler method with random stepsizes,
we can thus summarize our findings.
\begin{itemize}
\item All computed solutions remain in a $O(\tau)$-neighborhood 
of $\Sigma$ as long as $\Sigma$ is attractive.  This confirms that, on an ideal system, sliding 
should be taking place along $\Sigma$.  
\item All computed solutions move away from 
$\Sigma$ when they reach a sufficiently small neighborhood of a potential exit point.
\end{itemize}

There are important caveats to our results.
\begin{itemize}
\item
If the discontinuous dynamical system is the 
``idealization'' of a known smooth system, the approach described above may be 
misleading if we want to understand what happens to the solutions of the original
system.  For example, if the original system is given by \eqref{bilinear_reg},
as the regularization parameter goes to zero, one may obtain a limiting behavior which differs 
from the behavior of the discontinuous system.
What this means is that 
in this case the discontinuous model \eqref{GeneralPWS} is not a good model; 
as an example of this, see \cite{Polynikis.Hogan.diBernardo}. 
\item  Another class of systems that does not fit 
our analysis is given by control systems with non linear controls for which
the Filippov convex combination would not contain all the neighboring values of the 
vector fields, and hence our Euler's approximations of \eqref{GeneralPWS}
are not sufficient to understand the behavior of solutions (see \cite{Utkin.2} 
for a general theory).  
\end{itemize}

\section{Numerical experiments}\label{Numerics_section}

This section presents results of numerical simulations on several examples in $\R^3$ and $\R^4$. 
On each example, we compare the results of different experiments carried out with some of
the following strategies.

\begin{itemize}
\item[1)]
{\emph{Random Euler}}.  This refers to the approximation
of \eqref{GeneralPWS} computed with Euler method with random stepsizes as discussed in Section 
\ref{Euler_section}.  
\item[2)]
{\emph{Deterministic Euler}}.  This is the fixed stepsize implementation of Euler method
directly on \eqref{GeneralPWS} (see item (a) at p. \pageref{Possibilities}).
\item[3)]
{\emph{Regularized integration}}.  This refers to the approximation of
the bilinear regularized  vector field as proposed in Section \ref{Br_section}, by integrating
the regularized system  with the {\tt Matlab} built in functions
{\tt ode45}, and/or{\tt ode23s}, and/or {\tt ode15s}.\footnote{{\tt ode45} is a
Runge-Kutta integrator, suitable for non-stiff problems, {\tt ode23s} is a second order integrator
based on Rosenbrok formulas, better suited for stiff problems, and {\tt ode15s} is a variable
order (1 through 5) integrator based on the backward-differentiation-formulas (BDF), also suited
for stiff problems.}  All of these are variable stepsize integrators, where
local error control is enforced by a combination of relative and absolute
error tolerances ({\tt Reltol} and {\tt Abstol}).  (It is surely possible that
different solvers may perform somewhat differently than those we have used, but we have
no reason to suspect that a different solver will alter the outcome of our observations below.)
\item[4)]
{\emph{Unregularized integration}}.  This refers to the approximation 
of \eqref{GeneralPWS} computed (as a discontinuous system) with the {\tt Matlab} built in 
functions {\tt ode45} and/or {\tt ode23s} and/or {\tt ode15s}.
\end{itemize}

The most noteworthy aspects confirmed by our
numerical experiments are the following.
The numerical solution computed with ``Random Euler'' remains close to $\Sigma$ 
as long as $\Sigma$ is locally attractive, and instead leaves any small neighborhood of $\Sigma$ when 
$\Sigma$ looses attractivity.  On the other hand, the approximations computed by the ``Regularized 
integration''  might remain in a neighborhood of $\Sigma$ even when $\Sigma$ is not attractive 
in agreement with the 
dynamics of \eqref{fast_system}.  Also, ``Deterministic Euler'' is not a foolproof
option, since it superimposes its own dynamics to the true dynamics of the underlying problem.
Finally, ``Unregularized integration'' may just fail when requiring 
stringent values of the error tolerances, and also occasionally producing totally erroneous
approximations, and 
cannot be taken as a trustworthy indicator of what should happen in our context. 

For our experiments, we will adopt the simplifications discussed in 
Section \ref{Euler_section}.  
As already anticipated in Section \ref{Euler_section}: the examples are in $\R^3$ and $\R^4$, 
the discontinuity surfaces are $\Sigma_j=\{ x \in \R^n, \,\, x_j=0\}$, $j=1,2$ and the $w^i_j$'s 
are all constants except for one of them that depends on $x \in \Sigma$.  
Moreover, with the exception of Example \ref{example_4}, 
in all the examples below the Filippov vector field is uniquely defined. 
This is not a restriction for our scopes, since the behavior of solutions
(whether or not for the regularized system) at potential exit 
points remains ambiguous, as we will see below.    
 
\begin{exam}[$\Sigma$ looses attractivity through a potential 
tangential exit point]\label{example_1}

Consider the following vector fields
\begin{equation}\label{example_tangential}
\begin{split}
f_1=\begin{pmatrix} \frac 14  \\ 2-\frac {0.9}4 -x_3^2-x_4^2 \\ -\frac 25 x_3+x_4 \\ 
-\frac 25 x_3+\frac 45 x_4 \end{pmatrix}, \,\ 
f_2=\begin{pmatrix} 1 \\ -0.3 \\   -\frac 25 x_3+x_4 \\ -\frac 25 x_3+\frac 45 x_4 \end{pmatrix}, \\
f_3=\begin{pmatrix} -1 \\ 0.9 \\  -\frac 25 x_3+x_4 \\ -\frac 25 x_3+\frac 45 x_4 \end{pmatrix}, \quad \,\
f_4=\begin{pmatrix} -0.25 \\ -0.15 \\ -\frac 25 x_3+x_4 \\ -\frac 25 x_3+\frac 45 x_4 \end{pmatrix}\ .
\end{split}
\end{equation}
$\Sigma$ is the $(x_3,x_4)$ plane and on it there is the unique
vector field $\dot x=\begin{pmatrix} 0\\0\\ -\frac 25 x_3+x_4 \\ -\frac 25 x_3+\frac 45 x_4 \end{pmatrix}$.
The circle $\gamma=\{(x_3,x_4) :\,\ x_3^2+x_4^2=2\}$ is a curve of potential tangential exit points
on $\Sigma$. The region inside $\gamma$ is attractive upon sliding. It is attractive upon 
sliding along $\Sigma_{1,2}^\pm$ for $x_3^2+x_4^2 \leq 2-\frac{0.9}{4}$ and it is attractive 
upon sliding along 
$\Sigma_1^\pm$ and $\Sigma_2^+$ for $2-\frac{0.9}4<x_3^2+x_4^2<2$, {\rm (}see \cite{DiElLo2} 
for theoretical studies of an attractive co-dimension $2$ surface $\Sigma$ when one of the 
$w^i_j$'s is zero{\rm )}. Outside $\gamma$, 
$f_{\Sigma_1^-}$ points away from $\Sigma$ so that $\Sigma$ is not attractive.
All solutions with 
initial condition inside $\gamma$ will eventually meet $\gamma$. 
We expect trajectories of \eqref{GeneralPWS} to leave $\Sigma$ when they reach $\gamma$ 
and to start sliding along $\Sigma_1^-$ in direction opposite to $\Sigma$. 
This behavior is well illustrated in the right plot of Figure \ref{S1minus_fig}, 
where we depict the average solution computed with ``Random Euler'' with $\tau=10^{-6}$.
Next, we report on four more experiments.
\begin{itemize}
\item[a)]  Random Euler.
We computed the average exit point for an ensamble of $100$ initial conditions with 
$x_1$ and $x_2$ uniformly distributed in $[-\tau, \tau]^2$ and $x_3^2+x_4^2=1.7$. 
We compute the exit point for every solution as described in Section \ref{Euler_section}. 
For $\tau=10^{-6}$, we obtain the average exit point $\bar x \simeq 2.0865$ with 
standard deviation $\simeq 0.0202$. In the left plot in Figure \ref{S1minus_fig}, we plot the 
average trajectory obtained with stepsize $\tau=10^{-6}$. The exit point for the 
average trajectory is $\bar x \simeq 2.08$.
\item[b)] Regularized Integration.
For the regularized vector field \eqref{bilinear_reg}, we take $\alpha$ and 
$\beta$ as in \eqref{alp_bet_eq}. The corresponding fast 
system \eqref{fast_system} has a unique asymptotically stable equilibrium in $(0,1)^2$, a stable node, up to 
$x_3^2+x_4^2 = \bar \rho\simeq 2.225$. The curve $x_3^2+x_4^2=\bar \rho$ is a curve of 
saddle node bifurcation values for the fast system and for $x_3^2+x_4^2>\bar \rho$ there are 
no equilibria in $(0,1)^2$.  
(For this example, the behavior of \eqref{fast_system} does not change by choosing different
functions $\alpha$ and $\beta$ in \eqref{alp_bet_eq}, as long as they satisfy conditions 
\eqref{conditions}).  
The behavior of the solution of the regularized system is well illustrated in the 
right plot in Figure \ref{S1minus_fig}. The solution of the regular problem is 
computed with the Matlab function {\tt ode23s} and 
{\tt RelTol=AbsTol=}$10^{-9}$. The continuous line in the plot is the first component 
of the solution, while the dashed line is the second component. They are both plotted in function 
of $z=x_3^2+x_4^2$. 
\begin{figure}\caption{Example \ref{example_1}. 
Left: Average trajectory obtained with Random Euler and $\tau=10^{-6}$.  
Right: Regularized Integration, with $\epsilon_1=\epsilon_2=10^{-5}$ 
and initial condition $[10^{-3} \,\ 10^{-3} \,\ 0 \,\ \sqrt{1.5}]$. }\label{S1minus_fig}
\begin{tabular}{cc}
\includegraphics[width=150pt]{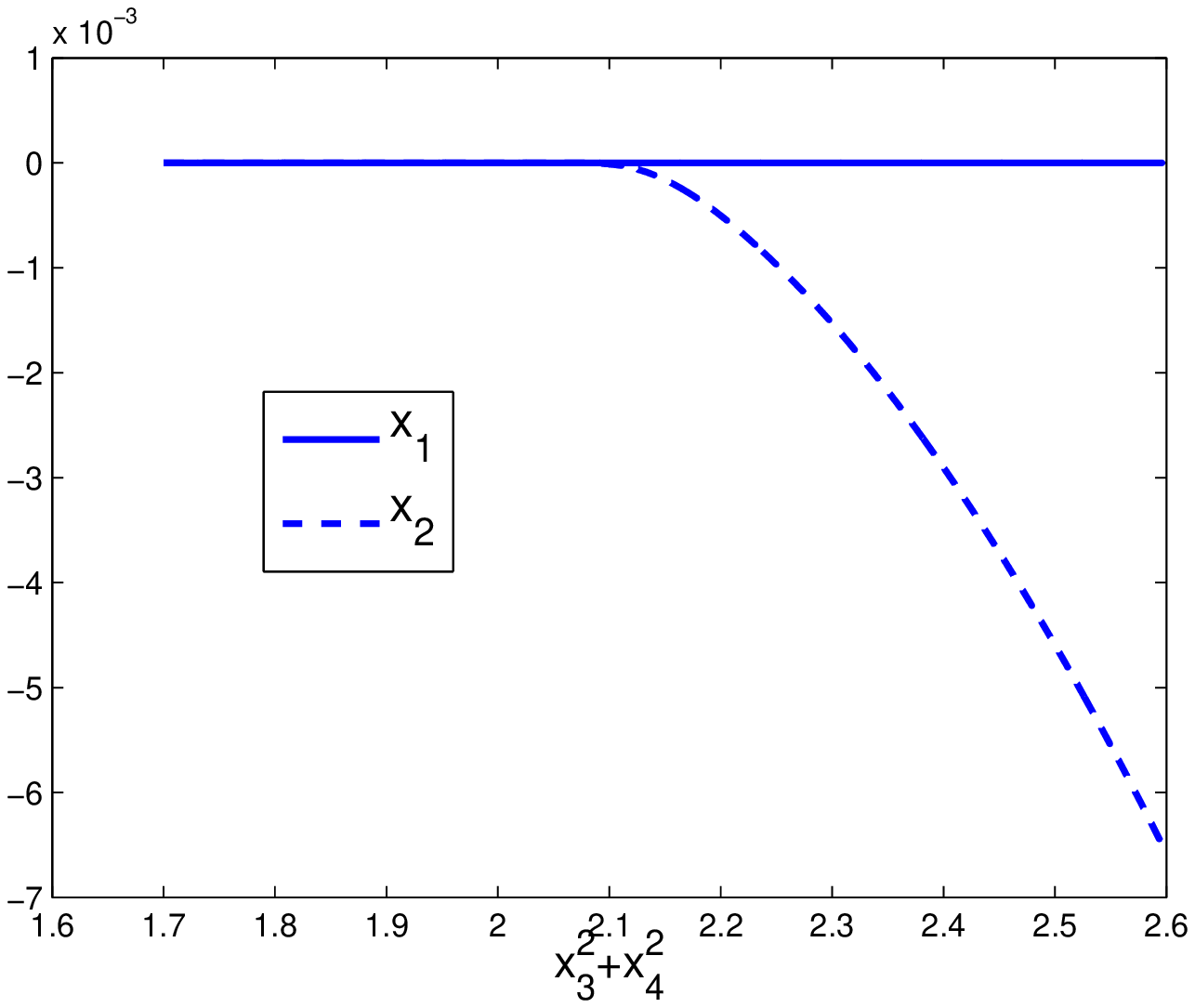} &
\includegraphics[width=150pt]{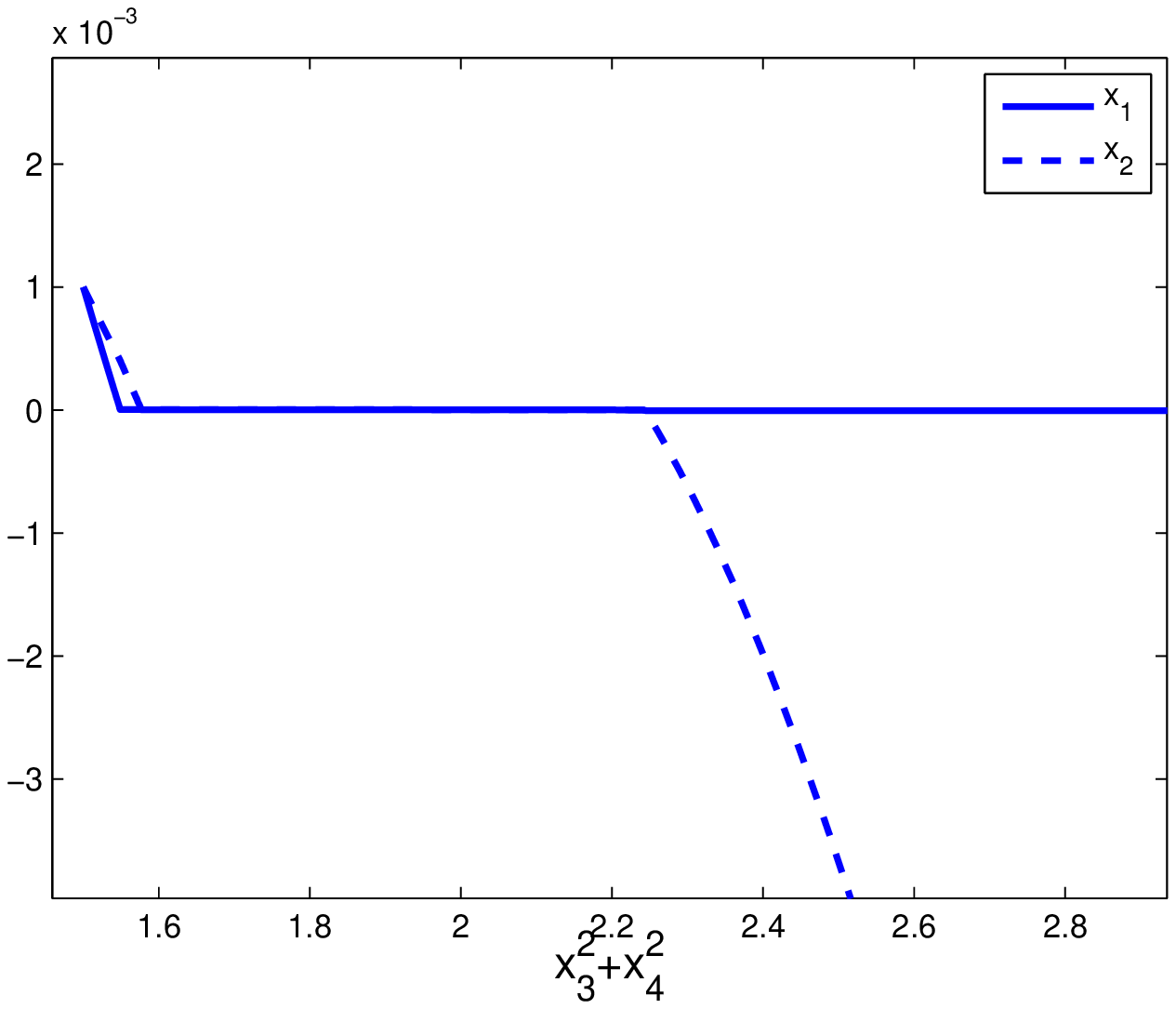} 
\end{tabular}
\end{figure}
\item[c)]
We also computed the solution of \eqref{GeneralPWS} with ``Deterministic Euler'' and fixed
stepsize $\tau=10^{-3}$, and 
initial condition $(10^{-4},10^{-4},0,\sqrt{1.7})$. First, 
we get rid of the transient and then plot the first two components of the approximation in 
Figure \ref{4d_S1minus_euler_po_fig}. None of the elements of the sequence generated by the
forward Euler's map is on $\Sigma_1$ or $\Sigma_2$ and hence the computed solution is 
always in one of the $R_j$'s and the selection of the vector field is straightforward. 
As it is clear from the plot, we do not recover the dynamics of the original system. 
The periodic orbit in the $(x_1,x_2)$ plane is spurious and is generated by forward Euler's own
dynamics. The periodic orbit survives past the curve of potential tangential exit points and 
indeed the third and fourth component of the plotted solution are such that 
$1.5 \leq x_3^2+x_4^2 \leq 30.817$. Notice also that the computed solution is 
trapped in a neighborhood of $\Sigma$ outside $\gamma$, when
$\Sigma$ is not attractive . Periodic motion persists also for smaller values of $\tau$,
though the orbit shrinks to the origin as $\tau\to 0$. 
\item[d)] Unregularized integration. Here, this approach works well for relaxed values of
the tolerances.  To witness, the numerical solution computed with {\tt ode23s} and 
{\tt RelTol}$=${\tt AbsTol}=$10^{-7}$, stays close to $\Sigma$ up to 
$x_3^2+x_4^2 \simeq 2.0014$. It then leaves 
$\Sigma$ to slide on $\Sigma_1^-$. With lower values of {\tt RelTol} and {\tt AbsTol}, 
the integrator takes more than half an hour in the time interval $[0, 0.5]$.  
The integrator {\tt ode45} shows a similar behavior.  
\end{itemize}
 
\begin{figure}
\caption{Example \ref{example_1}. Numerical solution computed with 
``Deterministic Euler'' and constant stepsize $\tau=10^{-3}$. }
\label{4d_S1minus_euler_po_fig}
\includegraphics[width=150pt]{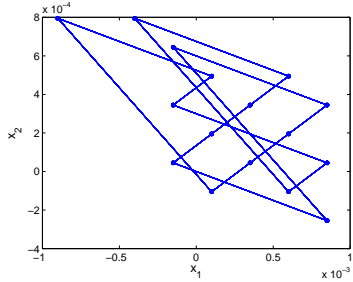}
\end{figure}  

\end{exam}

\vspace{.5cm}

\begin{exam}[$\Sigma$ looses attractivity through a
non tangential potential exit point]\label{example_2} 

The vector fields are
\begin{equation}\label{example_nontang}
f_1=\begin{pmatrix} (3-x_3)/5 \\ -1/5 \\ 1 \end{pmatrix}, \,\ 
f_2= \begin{pmatrix} 1/5 \\ -1/5 \\ 1 \end{pmatrix}, \,\ 
f_3=\begin{pmatrix} 1/5 \\ 2/5 \\ 1 \end{pmatrix}, \,\ 
f_4=\begin{pmatrix} -1 \\ -1/5  \\ 1 \end{pmatrix},
\end{equation}
and $\Sigma$ is the $x_3$-axis, with uniquely
defined sliding vector field $\dot x_3=1$.
$\Sigma$ is attractive through sliding along 
$\Sigma_1^+$ and $\Sigma_2^+$ 
for $x_3<3$. At $x_3=3$, $f_1$ is tangent to $\Sigma_1$ 
and points away from $\Sigma$ so that the point $(0,0,3)$ is 
a potential non tangential exit point. When $x_3>3$, the vector field $f_1$ 
points away both from $\Sigma_1$  and $\Sigma_2$ and we expect the solution of 
\eqref{GeneralPWS} to move away from $\Sigma$ and to enter $R_1$. 
\begin{itemize}
\item[a)]  Random Euler. This behaves qualitatively as we expect.
We take $100$ initial conditions uniformly distributed in 
$[-\tau,\tau]^2$.  With $\tau=10^{-4}$, the average ``exit point'' is $\bar x=2.9854$ 
with standard deviation $x_{s}=0.0122$. With $\tau=10^{-5}$, we obtain 
$\bar x=2.9923$, $x_{s}=0.0046$. The average trajectory obtained with stepsize $\tau=10^{-5}$ 
is plotted in Figure \ref{nontang_avg_traj_fig}. The exit point for the trajectory is $\bar x \simeq 2.9944$.
\begin{figure}\caption{System \eqref{example_nontang}. Average trajectory computed with Random Euler and 
stepsize $\tau=10^{-5}.$}\label{nontang_avg_traj_fig}
\includegraphics[width=150pt]{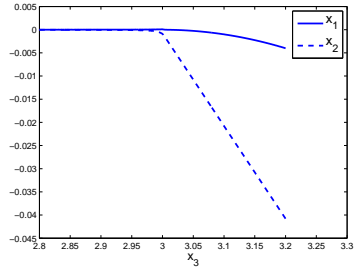}
\end{figure}

\item[b)] Regularized Integration.
For the regularized vector field \eqref{bilinear_reg}, we take $\alpha$ and $\beta$ as in 
\eqref{alp_bet_eq}, but we make different choices for the parameter $\epsilon$. 
We use two different parameters for 
$\alpha$ and $\beta$, namely $\epsilon_{\alpha}=\epsilon_{\beta}=10^{-6}$ and 
$\epsilon_{\alpha}=10^{-6}$, $\epsilon_{\beta}=10^{-5}$. The corresponding fast systems 
\eqref{fast_system} and \eqref{fast_system2} 
are not orbitally equivalent and, as a consequence of this, the behavior of the regularized 
solutions differs as well. We first consider the regularized system for $\epsilon_{\alpha}=\epsilon_{\beta}$. 
In this case, $(\alpha^*(x_3),\beta^*(x_3))$ is an asymptotically stable equilibrium of 
\eqref{fast_system} in $(0,1)^2$ for $x_3 < \bar x_3 \simeq 3.3853$. For $x_3=\bar x_3$, the eigenvalues 
of the Jacobian matrix of \eqref{fast_system} at the equilibrium cross the imaginary axis and 
$(\alpha^*(x_3),\beta^*(x_3))$ undergoes a subcritical Hopf bifurcation.  
In Figure \ref{nontang_fast_po_fig}, on the left, we plot both the stable equilibrium and the 
unstable periodic 
orbit for \eqref{fast_system} and $x_3=3.8$. The periodic orbit separates the
solutions that converge to the equilibrium $(\alpha^*,\beta^*)$ from the solutions 
that reach (possibly in finite time) the boundary of $[0,1]^2$. Notice that, 
if the initial condition of \eqref{full_system} satisfies (iii) of Theorem \ref{ClassicSingPert}, 
then b) in Theorem \ref{ClassicSingPert} follows, and the regularized solution remains in a 
neighborhood of $\Sigma$ as long as $(\alpha^*(x_3),\beta^*(x_3))$ is attractive. 
As a consequence of this reasoning,        
for $\epsilon$ sufficiently small, we expect the solution of \eqref{full_system} 
to remain close to $\Sigma$ up to $x_3=\bar x_3$. For $x_3>\bar x_3$, all solutions inside 
$(0,1)^2$ reach the boundary of the square, and we expect the regularized solution to 
move away from $\Sigma$. 
This is well illustrated in the left plot of Figure \ref{nontang_fig}. 
The equilibrium of \eqref{dummy} undergoes a subcritical Hopf bifurcation at 
$x_3=\hat x_3  \simeq 3.4476$. For $x_3> \hat x_3$, all solutions leave the square $[0,1]^2$. 
In the right plot of Figure \ref{nontang_fast_po_fig}, we show the unstable periodic orbit 
together with the stable equilibrium for \eqref{dummy} with $x_3=3.447$.   

\begin{figure}\caption{Example \ref{example_2}. 
Left: unstable periodic orbit and equilibrium $(\alpha^*,\beta^*)$ of \eqref{fast_system}
for $x_3=3.8$. 
Right: unstable periodic orbit and equilibrium of \eqref{dummy} for $x_3=3.447$.}
\label{nontang_fast_po_fig}
\includegraphics[width=150pt]{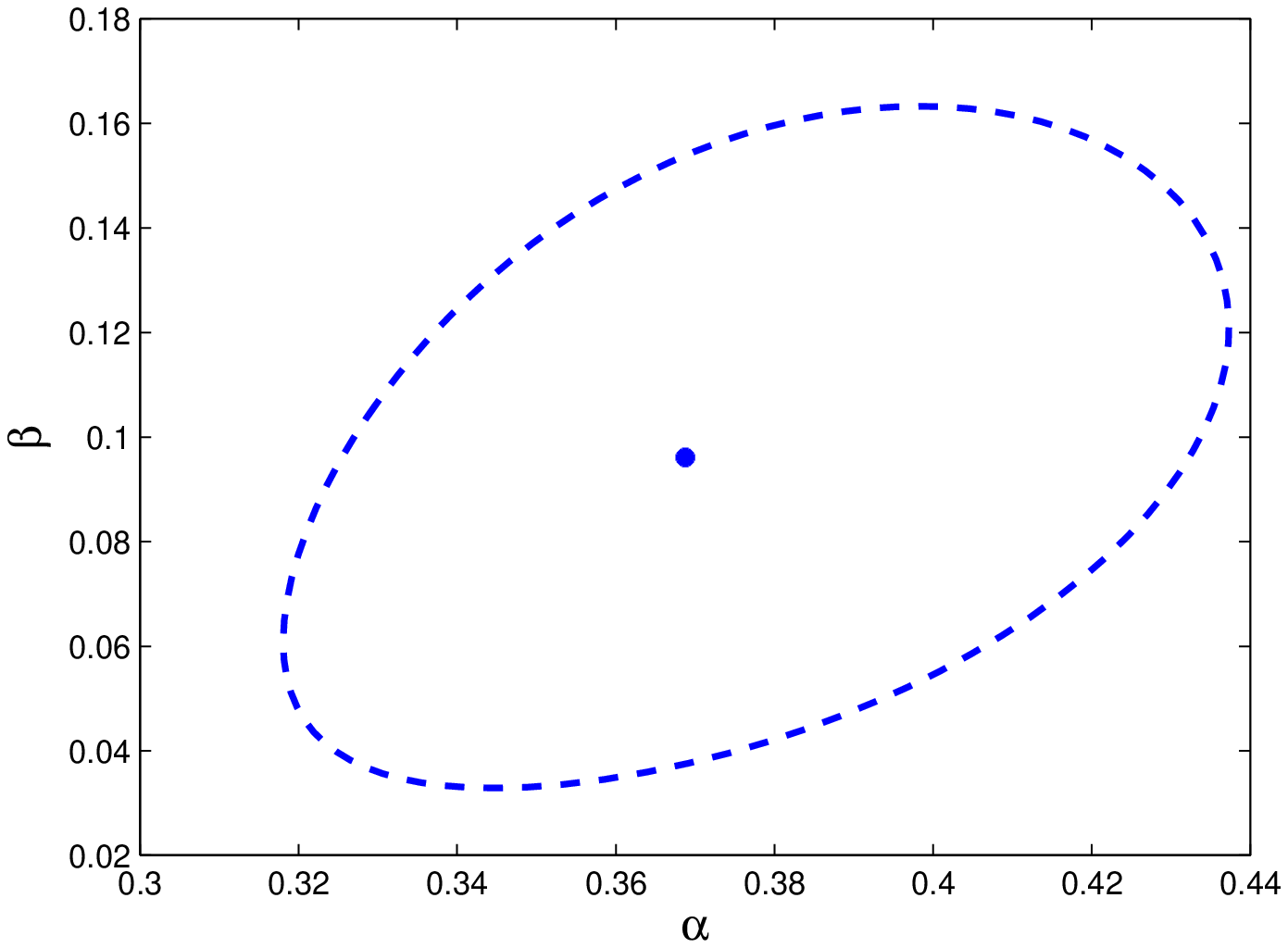}
\includegraphics[width=150pt]{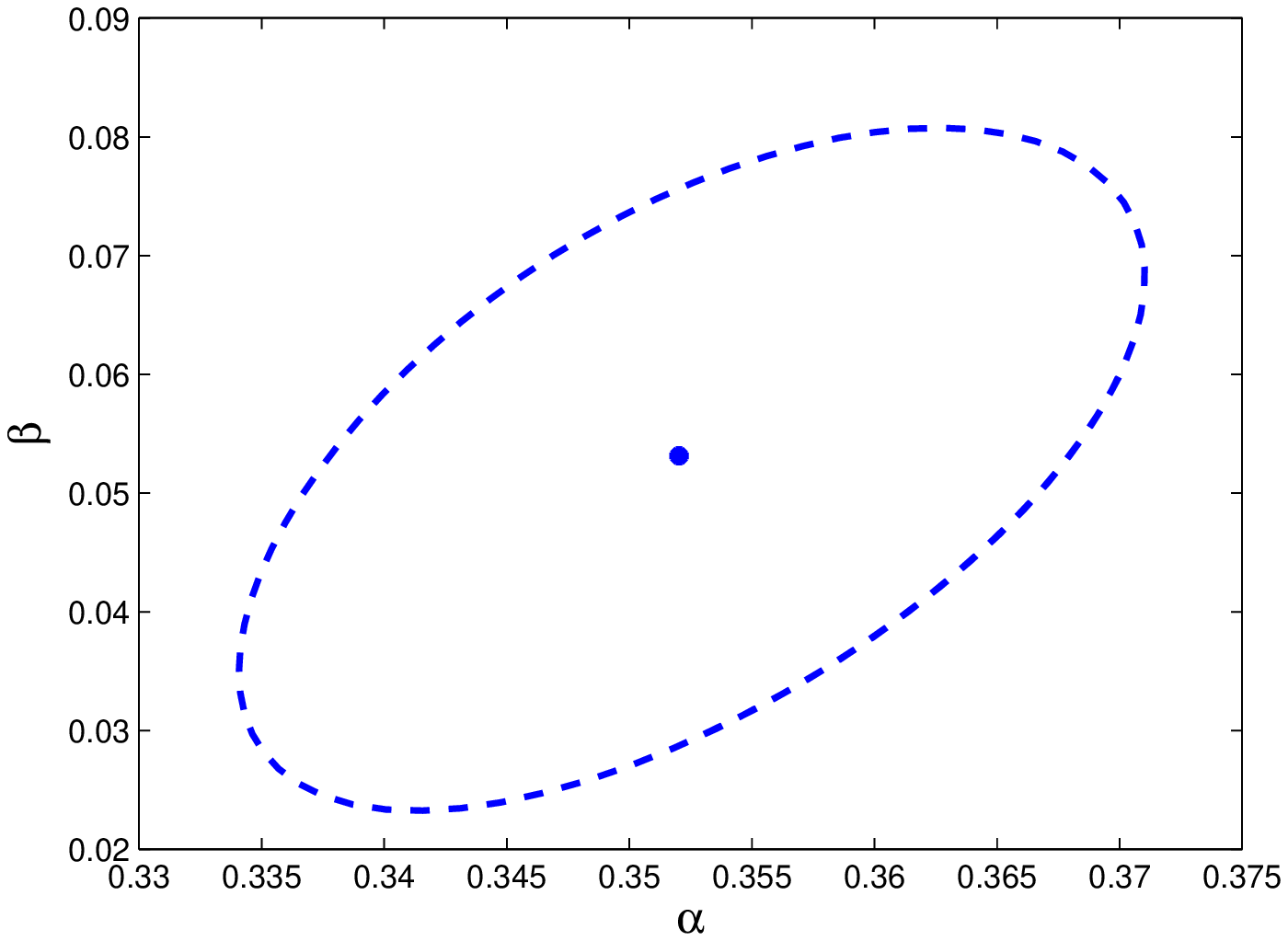}
\end{figure}
With the second set of parameter values, $\epsilon_{\alpha}=10^{-5}$ and 
$\epsilon_{\beta}=10^{-6}$, the dynamics of \eqref{fast_system2} is similar, 
but the bifurcation values are different.   
The equilibrium $(\alpha^*(x_3),\beta^*(x_3))$ of \eqref{fast_system2} 
is stable up to $x_3=\bar x_3 \simeq 3.2296$ and at $x_3=\bar x_3$ it undergoes a subcritical 
Hopf bifurcation. In the left plot of Figure \ref{nontang_fast_dynamics} we plot both the 
unstable periodic orbit and the stable equilibrium of \eqref{fast_system2} 
for $\epsilon_{\alpha}=10^{-6}$, $\epsilon_{\beta}=10^{-5}$ and 
$x_3=3.22$.  We also integrate the fast system backward in time for $x_3=3.21$ 
from an initial condition in a  neighborhood of the equilibrium. 
There is no periodic orbit in this case, since the backward solution meets the boundary 
of $[0,1]^2$ in finite time. The star in the plot marks the equilibrium $(\alpha^-,0)$ 
of the fast system.  This is the sliding vector field on $\Sigma_1^-$. 
\begin{figure}\caption{Example \ref{example_2}. 
Left Plot: unstable periodic orbit and stable equilibrium for 
the fast system \eqref{fast_system2}, $\epsilon_{\alpha}=10^{-6}$, $\epsilon_{\beta}=10^{-5}$ 
and $x_3=3.22$. Right plot: backward trajectory for the fast system for $x_3=3.21$. }
\label{nontang_fast_dynamics}
\includegraphics[width=150pt]{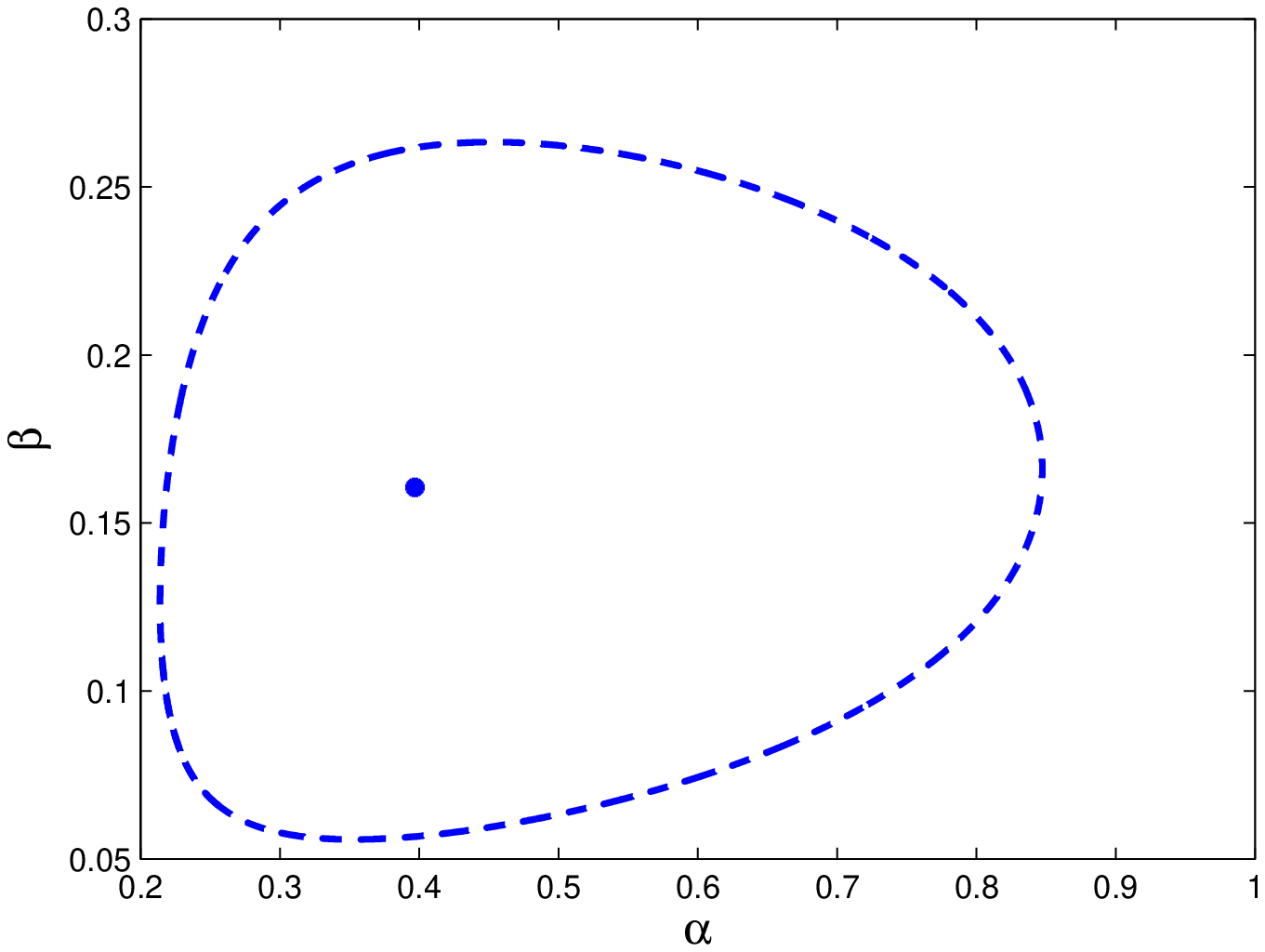}
\includegraphics[width=150pt]{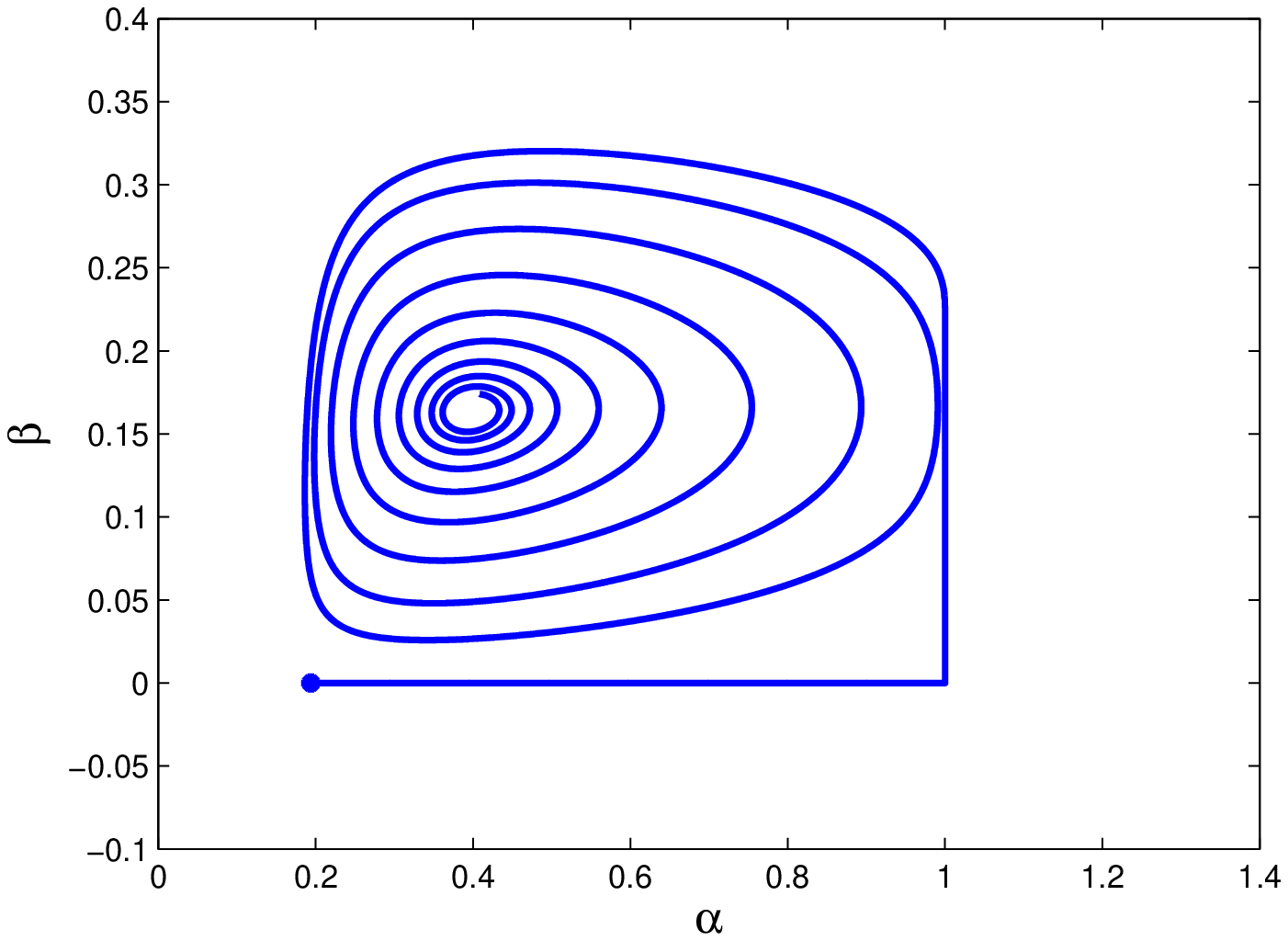}
\end{figure} 

On  the left of Figure \ref{nontang_fig}, we plot the numerical solution obtained 
integrating the regularized problem \eqref{bilinear_reg}, $\dot x=f_B^\epsilon(x)$,
with {\tt ode23s} and {\tt RelTol=AbsTol}$=10^{-12}$ for 
$\epsilon_{\alpha}=\epsilon_{\beta}=10^{-6}$.
The continuous line is the first component of the solution, while the dotted 
line is the second component. They are plotted against the third component. 
The solution stays close to $\Sigma$ up to $x_3 \simeq 3.39$ and then it enters $R_1$. 
This is in agreement with the singular perturbation analysis.
For the second set of parameters, $\epsilon_{\alpha}=10^{-6}$, $\epsilon_{\beta}=10^{-5}$,
the numerical solution obtained with {\tt ode23s} remains close 
to $\Sigma$ up to $x_3 \simeq 3.369$, and this is not in agreement with the singular perturbation 
analysis.   The numerical solution obtained with {\tt ode15s} is even less accurate and exits at 
$x_3 \simeq 3.485$.   We integrated the problem also with lower values of $\epsilon_\beta$,
while still having $\epsilon_{\alpha}=\frac{\epsilon_{\beta}}{10}$ so that the corresponding fast 
systems are all equivalent.  The numerical approximations are even less reliable, and indeed the
computed exit values increase as $\epsilon_\beta$ decreases.   On the other hand, 
the numerical solution computed with {\tt ode45} behave more reliable leaves $\Sigma$ to  
enter $R_1$ for $x_3 \simeq 3.238$, as we can see from the right plot in Figure \ref{nontang_fig}.

\begin{figure}\caption{Example \ref{example_2}. 
Solution of the regularized system \eqref{bilinear_reg} with {\tt ode45}.
Left: $\epsilon_\alpha=\epsilon_\beta=10^{-6}$. 
Right: $\epsilon_\alpha=10^{-6}$ and $\epsilon_\beta=10^{-5}$. }\label{nontang_fig}
\begin{tabular}{cc}
\includegraphics[width=150pt]{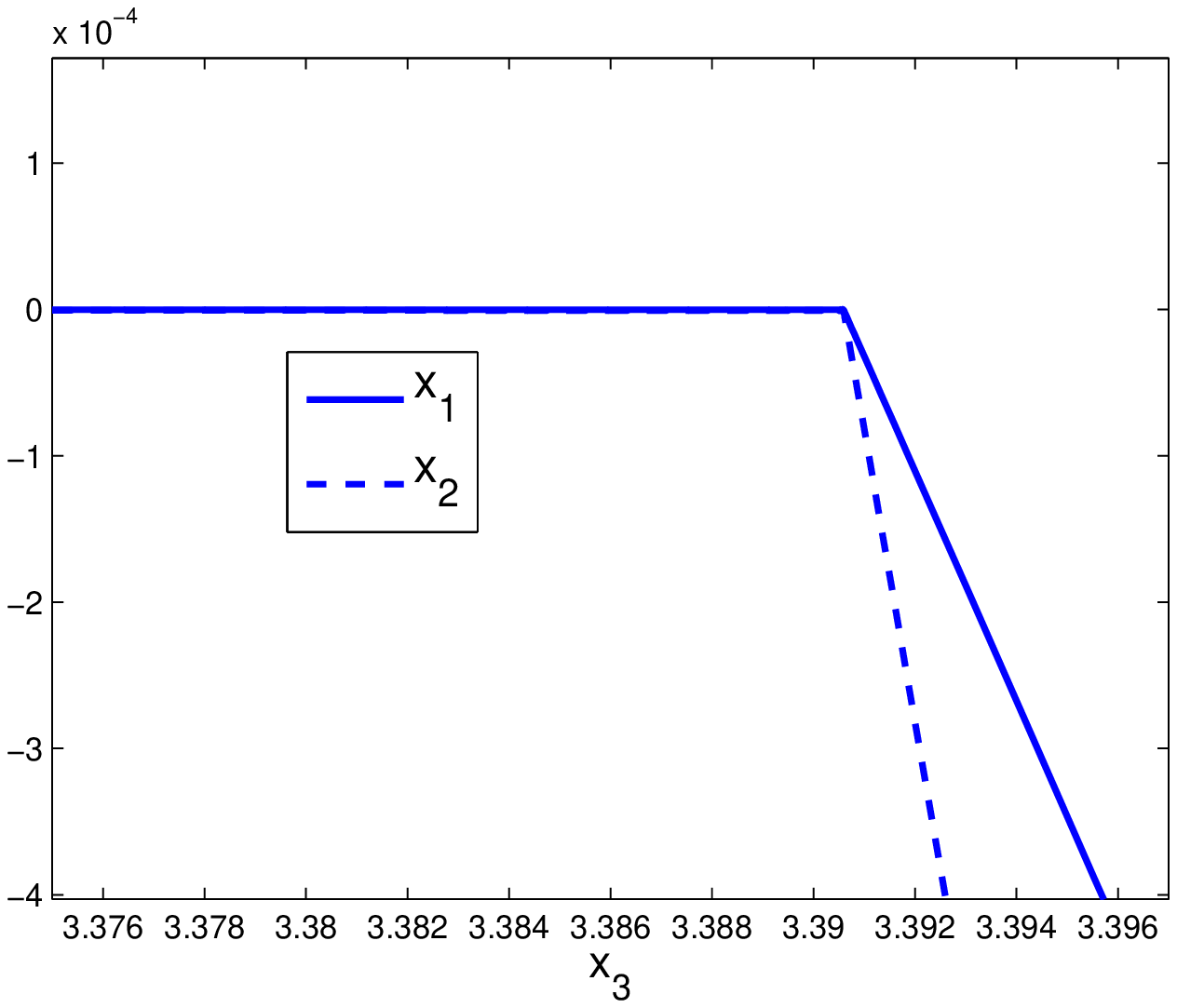} &
\includegraphics[width=150pt]{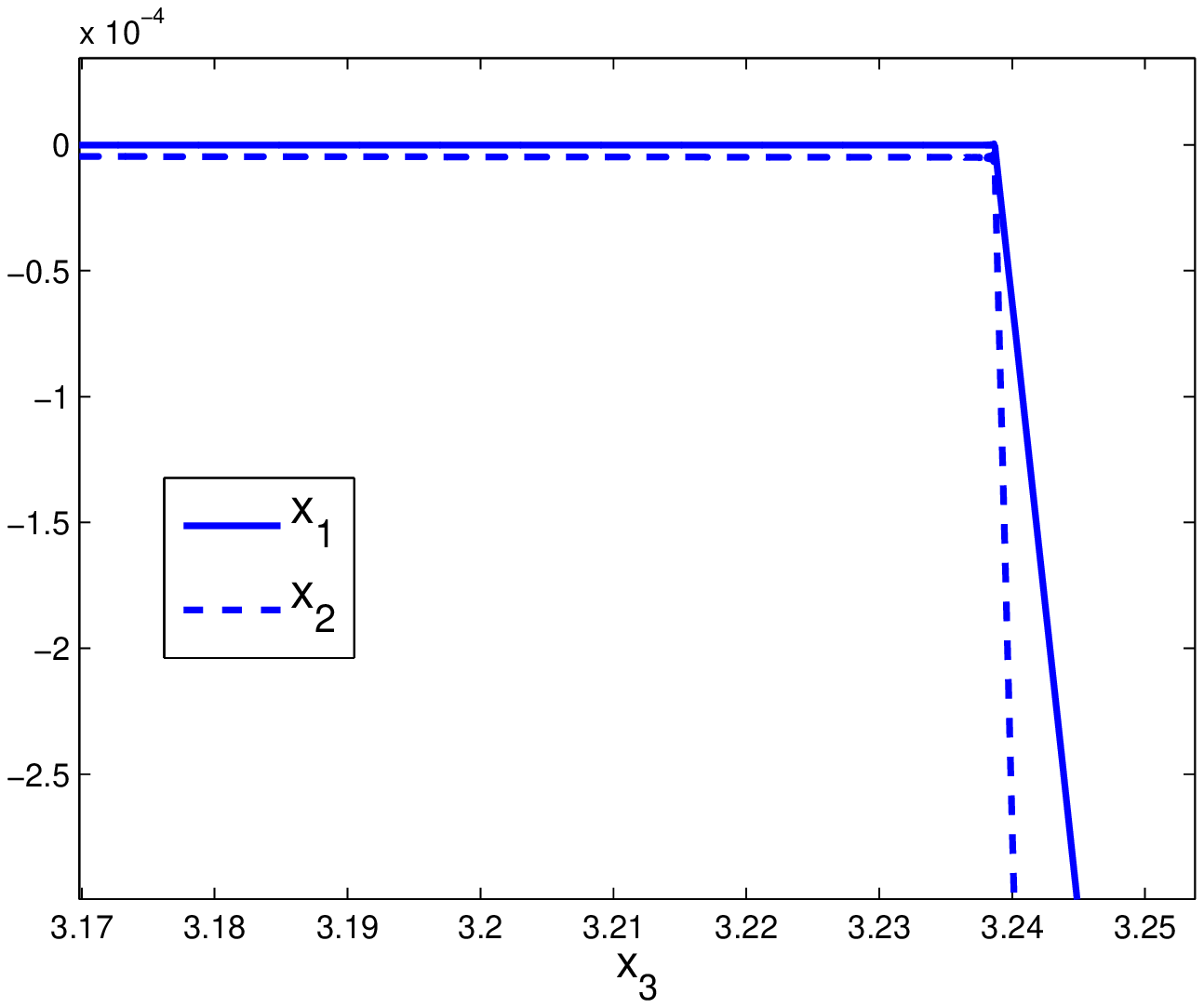} 
\end{tabular}
\end{figure}
\end{itemize}

\end{exam}

\vspace{0.5cm}

\begin{exam}[Spiral dynamics around $\Sigma$]\label{example_3}

In this example the vector fields are 
\begin{equation}\label{example_spiral}
f_1=\begin{pmatrix} \frac 13 \\ -\frac{x_3}3 \\ 1 \end{pmatrix}, \,\ 
f_2=\begin{pmatrix} -\frac 23 \\ -1 \\ 1 \end{pmatrix}, \,\ 
f_3=\begin{pmatrix} \frac 13 \\ \frac 23 \\ 1 \end{pmatrix}, \,\ 
f_4=\begin{pmatrix} -\frac 13 \\ 1  \\ 1 \end{pmatrix}.
\end{equation}
$\Sigma$ is the $x_3$-axis, with uniquely defined Filippov sliding motion $\dot x_3=1$,
and there is spiral like dynamics around $\Sigma$. For 
$x_3<1$, $\Sigma$ is attractive in finite time while, for $x_3>1$, $\Sigma$ is not locally 
attractive. 
\begin{itemize}
\item[a)] Random Euler.
Approximations computed with ``Random Euler''  move away from $\Sigma$ for $x_3>1$. 
In Figure \ref{spiral_fig} we show the average trajectory computed with 
$\tau=10^{-5}$ and initial condition $(10^{-6},10^{-6},0.5)$.  
\begin{figure}\caption{Example \ref{example_3}. 
Average trajectory obtained with Random Euler and $\tau=10^{-5}$.}
\label{spiral_fig}
\begin{tabular}{cc}
\includegraphics[width=150pt]{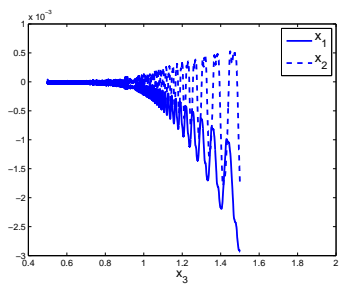} &
\end{tabular}
\end{figure}
\noindent For a statistical estimate of the exit point, we consider an ensamble of $100$ 
initial conditions with the first two components uniformly distributed in 
$[-\tau,\tau]^2$. For $\tau=10^{-5}$, 
the mean value of $x_3$ at the exit point is $\bar x_3 \simeq 0.94777$, with 
standard deviation $\simeq 0.04297$.
For $\tau=10^{-6}$, the mean value of $x_3$ at the exit point is $\bar x_3 \simeq 0.97910$ and 
the standard deviation is $\simeq 0.00529$.  
In Figure \ref{spiral_fig2}, on the left, we plot a trajectory 
of \eqref{GeneralPWS} obtained with Random Euler and stepsize $\tau=10^{-6}$.
On the right of Figure \ref{spiral_fig2} we plot the $2$-norm of the vector 
$(h_1,h_2)$ (for us, this is the vector $(x_1,x_2)$)
in function of $x_3$. 
The ``X'' in the plot marks the estimated exit point. 

\begin{figure}\caption{Example \ref{example_3}. 
Left: a solution of \eqref{example_spiral} obtained with Random Euler and $\tau=10^{-5}$. 
Right: norm of the vector $(h_1,h_2)$ in function of $x_3$ for this trajecytory.}
\label{spiral_fig2}
\begin{tabular}{cc}
\includegraphics[width=150pt]{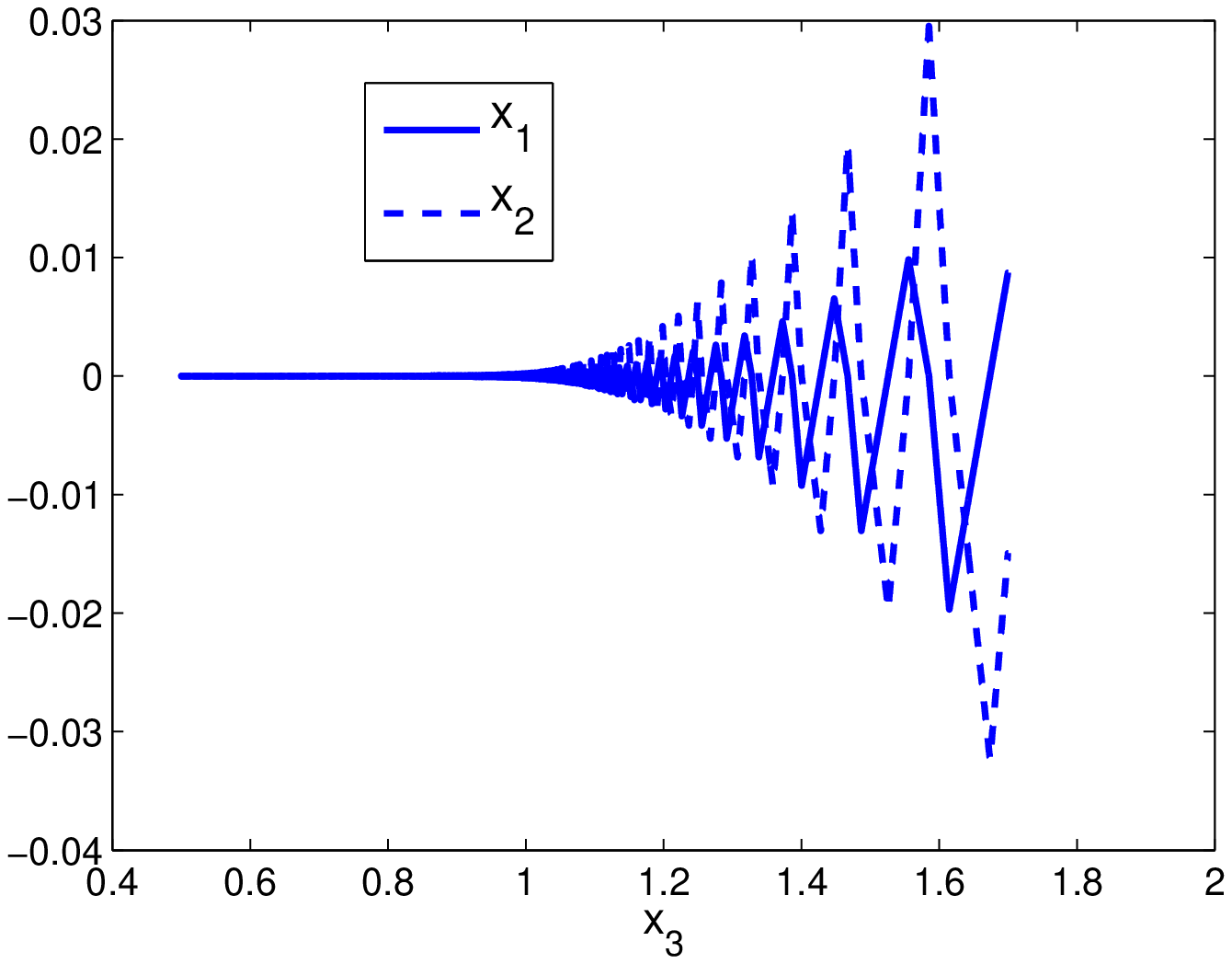} &
\includegraphics[width=150pt]{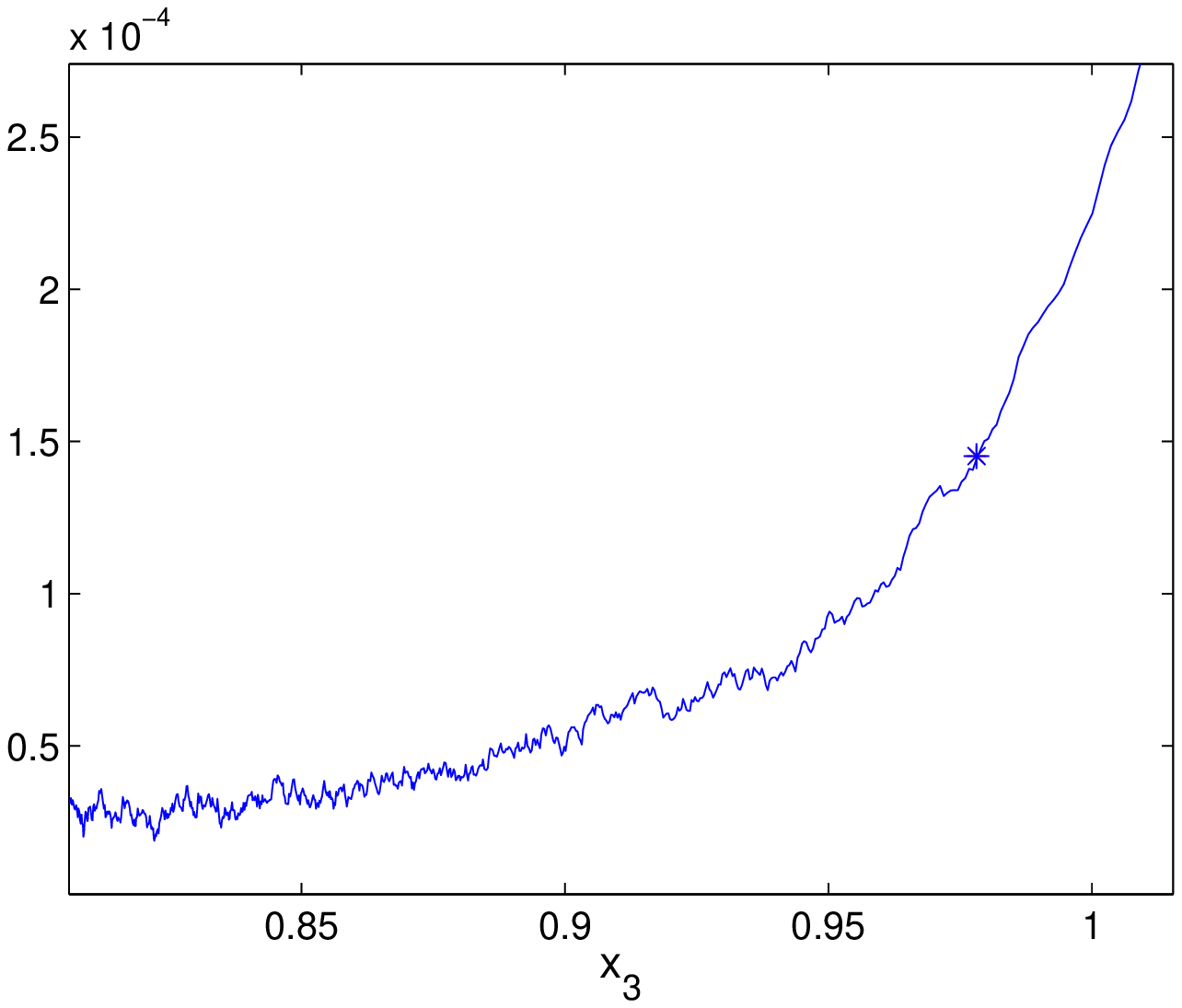}
\end{tabular}
\end{figure}
\item[b)]  Regularized Integration.
Next, consider integrating the regularized vector field \eqref{bilinear_reg} with 
$\alpha$ and $\beta$ as in \eqref{alp_bet_eq} and two different sets of parameters. 
We first consider 
$\epsilon_{\alpha}=\epsilon_{\beta}=10^{-4}$. For these parameters values, the equilibrium 
$(\alpha^*,\beta^*)$ of the fast system \eqref{fast_system} is stable up to $x_3 \simeq 1.41798$, 
then the eigenvalues of the Jacobian matrix at the equilibrium cross the imaginary axis. 
We compute the solution of the regularized system with initial condition 
$(10^{-4},10^{-4},0)$ in the time interval $[0,2]$ (so that $0 \leq x_3 \leq 2$) with 
{\tt ode23s}, {\tt ode15s} and {\tt ode45} with {\tt RelTol=AbsTol}$=10^{-12}$. 
The computed solutions show different behavior: the numerical solution computed with
{\tt ode23s} and {\tt ode15s} remain in a small neighborhood of $\Sigma$ for 
the whole time interval $[0,2]$, 
while the solution computed with {\tt ode45} moves away from $\Sigma$ for $x_3 \simeq 1.5$ 
as it is seen from the left plot of Figure \ref{spiral_ode23s_fig}. 
The average stepsize used by the stiff integrators is $O(10^{-2})$, 
while the one used by the explicit integrator is $O(10^{-4})$.
We then consider a second set of parameters $\epsilon_{\alpha}=10^{-4}$ and 
$\epsilon_{\beta}=10^{-3}$. The equilibrium $(\alpha^*,\beta^*)$  of the corresponding 
fast system \eqref{fast_system2} is unstable when $x_3>0$. 
On the right of Figure \ref{spiral_ode23s_fig} we plot the approximation computed with 
{\tt ode23s} (the approximation computed with {\tt ode45} behaves in a similar way). 
\begin{figure}
\caption{Example \ref{example_3}. Regularized integration with {\tt ode23s}.
Left: $\epsilon_{\alpha}=\epsilon_{\beta}=10^{-4}$.
Right: $\epsilon_{\alpha}=10^{-4}$ and
$\epsilon_{\beta}=10^{-3}$.}
\label{spiral_ode23s_fig}
\begin{tabular}{cc}
\includegraphics[width=150pt]{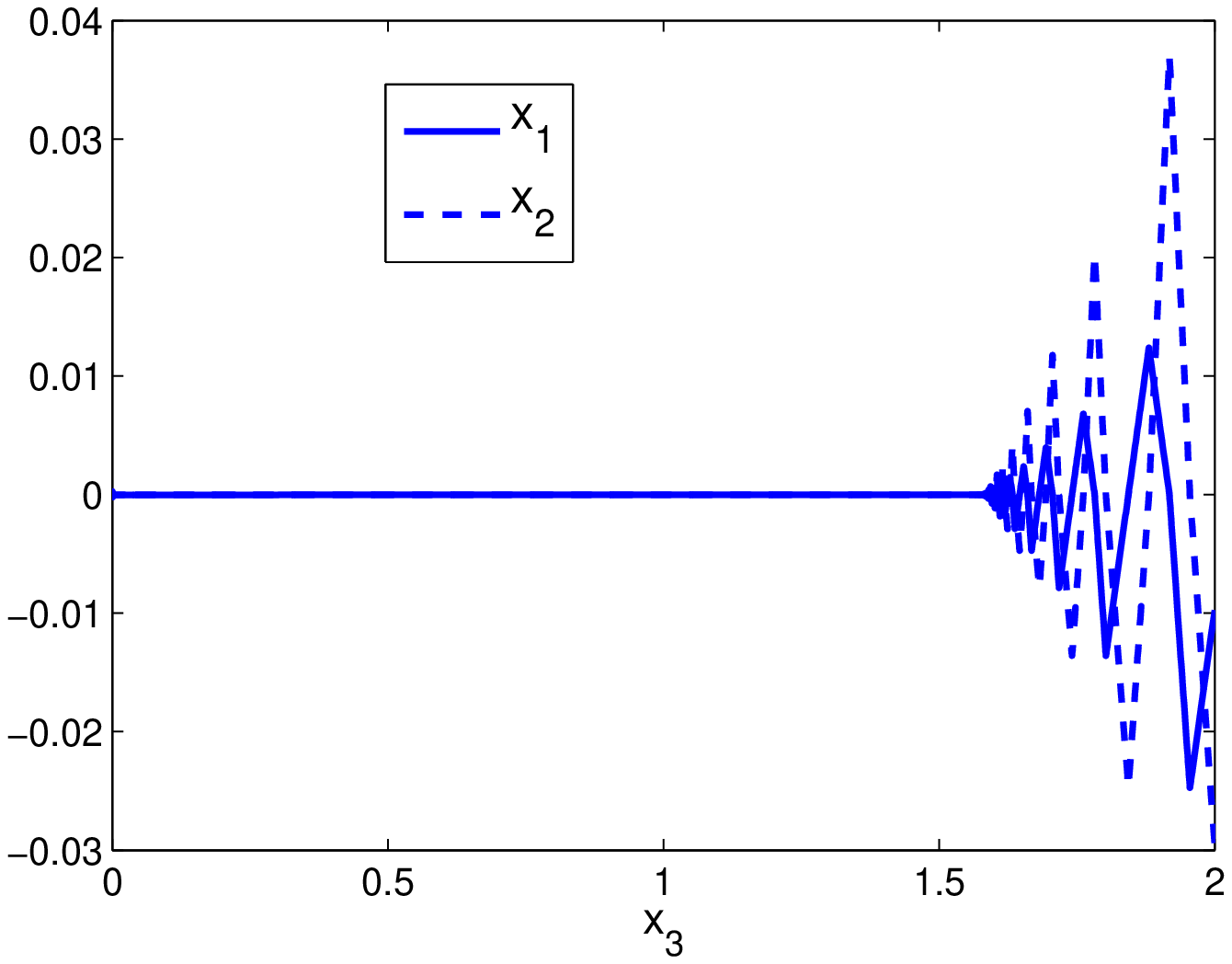} 
\includegraphics[width=150pt]{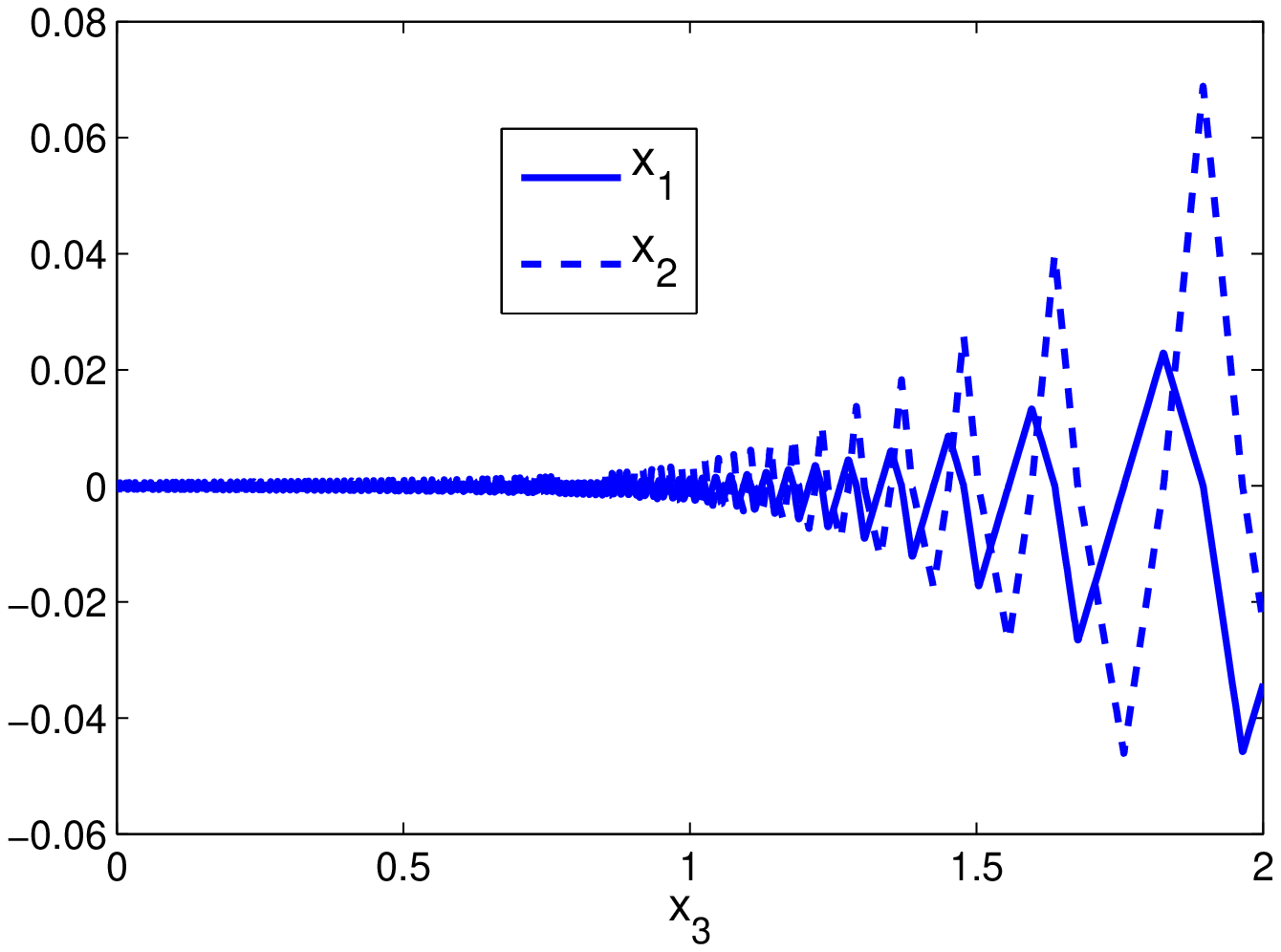} 
\end{tabular}
\end{figure}
\item[c)] Unregularized integration. The numerical solution obtained with the stiff 
Matlab integrator {\tt ode23s}, with
{\tt RelTol}$={\tt AbsTol}=10^{-7}$, stays close to $\Sigma$ up to $x_3=1$ and then it leaves 
$\Sigma$ to enter one of the $R_j$'s. For lower values of {\tt RelTol}, 
the integrator takes more than half an hour 
in the time interval $[0,1]$. The numerical solution obtained with {\tt ode15s} and 
{\tt RelTol}={\tt AbsTol}=$10^{-6}$ 
is very inaccurate as it is evident from the plot in Figure \ref{ode15s_fig}. 
Lower tolerances do not produce better results.
\begin{figure}
\label{ode15s_fig}
\caption{Example \ref{example_3}.  Unregularized integration with {\tt ode15s} and 
{\tt RelTol}={\tt AbsTol}=$10^{-6}$.}
\includegraphics[width=150pt]{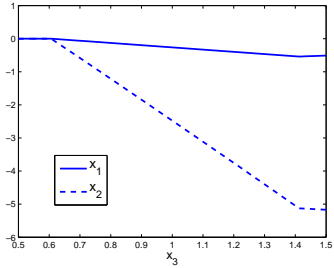}
\end{figure} 
\end{itemize}
\end{exam}

\vspace{.5cm}

\begin{exam}[Filippov's vector field on $\Sigma$ is ambiguous]\label{example_4}

We consider the following vector fields 
\begin{equation}\label{example_ambiguous}
\begin{split}
f_1=\begin{pmatrix} \frac 12 \\ 1 \\ -x_3+\frac 12 x_4 \\ x_4 \end{pmatrix}, \qquad \qquad 
        f_2=\begin{pmatrix} 1 \\ \frac 12 \\ -x_3+\frac 12 x_4 \\ x_4 \end{pmatrix} \\
       f_3=\begin{pmatrix} -(x_3-3)^2-(x_4-3)^2+5 \\ 1 \\ -x_3+28x_4 \\ x_4 \end{pmatrix}, \,\
        f_4= \begin{pmatrix} -1 \\ -1 \\ -x_3+4x_4 \\ x_4 \end{pmatrix},
\end{split}
\end{equation}  
and $\Sigma$ is the $(x_3,x_4)$ plane.
The circle $\gamma=\{x \in \R^4, (x_3-3)^2+(x_4-3)^2=4 \}$ (see Figure \ref{ambiguous_fig}), 
divides $\Sigma$ in two regions.  Outside $\gamma$, $\Sigma$ is attractive upon 
sliding along $\Sigma_1^+$ and $\Sigma_2^+$. 
For all points on $\gamma$, the vector field $f_{\Sigma_2}^+$ is tangent to $\Sigma$ and, 
inside $\gamma$, it points away from $\Sigma$, so that $\Sigma$ is not attractive from inside \
$\gamma$. We would expect the solution of \eqref{GeneralPWS} to leave $\Sigma$ once it reaches $\gamma$.  
Note that there is a family of Filippov sliding vector fields on $\Sigma$, namely:
$\dot x_1=\dot x_2=0$, $\dot x_3=-x_3+(16-13 \lambda)x_4$, $\dot x_4=x_4$, with 
$0 \leq \lambda \leq \frac 27$. 
\begin{itemize}
\item[(a)]  Random Euler.
The ambiguity of a Filippov sliding vector field is
clearly reflected in the numerical solutions computed with Random Euler.
In Figure \ref{ambiguous_fig} we plot the $x_3$ and $x_4$ components
of $1000$ trajectories computed with
$\tau=10^{-2}$, and same initial condition $(10^{-2},10^{-2},3,0.9)$. 
The dotted circle in the plot is the curve $\gamma$. The two bold darker lines are two 
sample trajectories. The shaded region is obtained by plotting all $1000$ trajectories. 
The plot suggests that the choice of random stepsizes
covers the region obtained by choosing one of the possible
vector fields in Filippov's differential inclusion.   
\begin{figure}
\caption{Example \ref{example_4}: plot in the plane $(x_3,x_4)$ of
$1000$ approximations obtained with Random Euler and 
$\tau=10^{-2}$, and initial condition $(10^{-2},10^{-2},3,0.9)$.}
\label{ambiguous_fig}
\includegraphics[width=150pt]{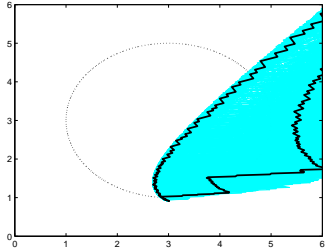}
\end{figure}
For completeness, in Figure \ref{ambiguous_one_traj_fig} we plot in function of 
time the first and second component of the average trajectory computed with Random Euler 
with $\tau=10^{-5}$. At $t \simeq 0.1067$, the third and fourth components of the average trajectory 
are on $\gamma$ and indeed the plotted solution leaves 
$\Sigma$ and starts sliding on $\Sigma_2^+$ at $t \simeq 0.12$.    
\begin{figure}\caption{Example \ref{example_4}. First and second component of 
the numerical solution computed with Random Euler with $\tau=10^{-5}$.}
\label{ambiguous_one_traj_fig}
\includegraphics[width=150pt]{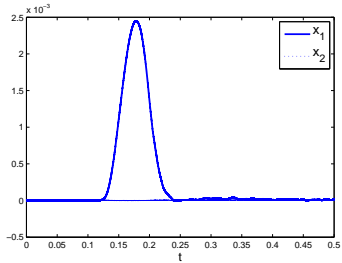}
\end{figure}
\item[b)]
Regularized Integration.
Here, the computed approximations of
the regularized system move away from $\Sigma$ once they reach $\gamma$. 
We note that $\gamma$ is a curve of saddle-node bifurcation values for 
\eqref{fast_system} for any parameter value $\epsilon_{\alpha}$ and/or $\epsilon_{\beta}$. 
On $\gamma$, $(\alpha^*,\beta^*)=(1,0.5)$ is a double root of \eqref{algebraic_eq}, while inside $\gamma$ 
there is no solution of \eqref{algebraic_eq}. 
\end{itemize}
\end{exam}

\section{Conclusions}\label{Concl}
In this work we have been interested in studying the behavior of solutions of piecewise smooth 
systems in the neighborhood of a co-dimension $2$ discontinuity surface $\Sigma$, 
intersection of two co-dimension $1$ discontinuity surfaces.
It has long been accepted that if solution trajectories cannot leave $\Sigma$ ($\Sigma$ is
attractive), some form
of sliding motion on $\Sigma$ should be taking place.  Precisely which sliding motion
has been the subject of much investigation, but it has not been our concern in this paper.
Our chief interest in this work has been trying to understand what should happen 
when $\Sigma$ loses attractivity (at generic first order exit points).  To our knowledge, 
this type of study had not been carried out before.

We took the point of view that the piecewise smooth system \eqref{GeneralPWS} was the
only information at our disposal, and treated this model with its own mathematical dignity.
Naturally, if \eqref{GeneralPWS} arose as a simplified
model for some other known differential system, then this original system
should ultimately guide the search for appropriate dynamics near $\Sigma$, and it may well
be that the dynamics of this ``true'' system are not matched by
those of \eqref{GeneralPWS}.  If this is the case, we should legitimately question the validity
of the model \eqref{GeneralPWS} in the first place.  On the other hand, in the absence of 
knowledge of an underlying ``true'' system, when \eqref{GeneralPWS} is the only datum we have, 
then we believe that we should try to modify this model so that the dynamics of the
modified system match those of \eqref{GeneralPWS}. 
 
To obtain information on the dynamics of \eqref{GeneralPWS}, we
proposed use of a simple Euler method with random steps, uniformly chosen with
respect to a reference, small, stepsize.  Our study unambiguously show that:
(i) when $\Sigma$ is attractive, solution trajectories remain near
$\Sigma$ (thereby validating an idealized sliding motion on $\Sigma$);
(ii) when $\Sigma$ loses attractivity, solution trajectories leave
a neighborhood of $\Sigma$.

Several other possibilities have also been considered in this work: regularization
techniques, plain and simple Euler method with fixed stepsize, and direct numerical
integration of \eqref{GeneralPWS} with sophisticated off-the-shelf solvers for differential 
equations.  None of these options satisfactorily resolved the dynamics of \eqref{GeneralPWS},
and often produced misleading behavior.  Ultimately, this occurred because each of these
choices either superimposed its own dynamics on those of \eqref{GeneralPWS} (as Euler method
and regularization techniques do, further producing different behaviors depending on
how the regularization is made), or just failed to produce reliable answers in too many cases
(this was the case with directly solving \eqref{GeneralPWS} with existing software,
where the outcome dramatically dependend on the solver used, or on the tolerances values,
or both).

Unfortunately, our conclusions are not fully satisfactory either.  Our analysis
tells us that is the dynamics of \eqref{GeneralPWS} around $\Sigma$ that
must be used to tell us what should happen in a neighborhood of $\Sigma$, but we know of
no general foolproof mean to regularize the system so that the regularized trajectory 
will be following the dynamics of \eqref{GeneralPWS}.  Perhaps, and --again-- as long
as the model \eqref{GeneralPWS} is appropriate, the most reliable and practically efficient
way to proceed is to accept some form of idealized sliding motion on $\Sigma$ as long as
$\Sigma$ is attractive, while also demanding that a sliding trajectory leaves $\Sigma$
when the latter loses its attractivity.  The construction of appropriate sliding
vector fields fulfilling these requests remains an outstanding and challenging task.

\end{document}